\documentclass[10pt,letterpaper,dvips]{article}

 \usepackage{amssymb,latexsym,amsmath,amsthm}
\newtheorem{thm*}{Theorem}

\newtheorem{thm}{Theorem}[section]

\newtheorem{dfn}{Definition}[section]

\newtheorem{exam}{Example}[section]

\newtheorem{lemma}{Lemma}[section]

\newtheorem{remark}{Remark}[section]

\newtheorem{cor}{Corollary}[section]

\usepackage{fullpage}

\begin{document}

\def\d{ \partial } 
\def\Na{{\mathbb{N}}}

\def\Z{{\mathbb{Z}}}

\def\IR{{\mathbb{R}}}

\def\L{ {\mathcal{L}}}

\newcommand{\E}[0]{ \varepsilon}

\newcommand{\s}[0]{ \mathcal{S}}

\newcommand{\AO}[1]{\| #1 \| }

\newcommand{\BO}[2]{ \left( #1 , #2 \right) }

\newcommand{\CO}[2]{ \left\langle #1 , #2 \right\rangle} 

\newcommand{\R}[0]{ \IR\cup \{\infty \} } 

\newcommand{\co}[1]{ #1^{\prime}} 

\newcommand{\p}[0]{ p^{\prime}} 

\newcommand{\m}[1]{   \mathcal{ #1 }}

\newcommand{ \A}[1]{ \left\| #1 \right\|_H }

\newcommand{\B}[2]{ \left( #1 , #2 \right)_H }

\newcommand{\C}[2]{ \left\langle #1 , #2 \right\rangle_{  H^* , H } }

 \newcommand{\HON}[1]{ \| #1 \|_{ H^1} }

\newcommand{ \Om }{ \Omega}

\newcommand{ \pOm}{\partial \Omega}

\newcommand{\D}{ \mathcal{D} \left( \Omega \right)} 

\newcommand{\Ov}{ \overline{ \Omega}}

\newcommand{\DP}{ \mathcal{D}^{\prime} \left( \Omega \right)  }

\newcommand{\DPP}[2]{   \left\langle #1 , #2 \right\rangle_{  \mathcal{D}^{\prime}, \mathcal{D} }}

\newcommand{\PHH}[2]{    \left\langle #1 , #2 \right\rangle_{    \left(H^1 \right)^*  ,  H^1   }    }

\newcommand{\PHO}[2]{  \left\langle #1 , #2 \right\rangle_{  H^{-1}  , H_0^1  }} 

 \newcommand{\HO}{ H^1 \left( \Omega \right)}

\newcommand{\HOO}{ H_0^1 \left( \Omega \right) }

\newcommand{\CC}{C_c^\infty\left(\Omega \right) }

\newcommand{\N}[1]{ \left\| #1\right\|_{ H_0^1  }  }

\newcommand{\IN}[2]{ \left(#1,#2\right)_{  H_0^1} }

\newcommand{\INI}[2]{ \left( #1 ,#2 \right)_ { H^1}} 

\newcommand{\HH}{   H^1 \left( \Omega \right)^* } 

\newcommand{\HL}{ H^{-1} \left( \Omega \right) }

\newcommand{\HS}[1]{ \| #1 \|_{H^*}}

\newcommand{\HSI}[2]{ \left( #1 , #2 \right)_{ H^*}}

\newcommand{\WO}{ W_0^{1,p}} 
\newcommand{\w}[1]{ \| #1 \|_{W_0^{1,p}}}  

\newcommand{\ww}{(W_0^{1,p})^*}

\newcommand{\labeld}[1]{ \mbox{ \qquad (#1)}  \qquad   \label{ #1}  }

 \noindent \textbf{   \Large{Optimal Hardy inequalities  for general  elliptic operators with \\ improvements}      }

\begin{center} Craig Cowan\footnote{As partial fulfillment of a doctorate degree.   }

\end{center}

\begin{center}Department of Mathematics \\ University of British Columbia \\ Vancouver, B.C., Canada, V6T 1Z2

\end{center}

\begin{abstract}

We establish Hardy inequalities of the form 
\begin{equation}
\label{ONE}  \int_\Omega | \nabla u|_A^2 dx \ge \frac{1}{4} \int_\Omega \frac{| \nabla E|_A^2}{E^2}u^2dx, \qquad u \in H_0^1(\Omega)
\end{equation}
 where $ E$ is a positive function defined in $ \Omega$, $ -div(A \nabla E)$ is a nonnegative nonzero finite measure in $ \Omega$ which we denote by $ \mu$ and where $ A(x)$ is a $ n \times n$ symmetric, uniformly positive definite matrix defined in $ \Omega$ with $ | \xi |_A^2:= A(x) \xi \cdot \xi$ for $ \xi \in \IR^n$.   We show that (\ref{ONE}) is optimal if $ E=0$ on $ \pOm$ or $ E=\infty$ on the support of $ \mu$ and is not attained in either case.    When $ E=0$ on $\pOm$ we show 
\begin{equation} \label{TWO}  \int_\Omega | \nabla u|_A^2dx \ge \frac{1}{4} \int_\Omega \frac{| \nabla E|_A^2}{E^2}u^2dx + \frac{1}{2} \int_\Omega \frac{u^2}{E} d \mu, \qquad u \in H_0^1(\Omega) 
\end{equation} is optimal and not attained.  
 Optimal weighted versions of these inequalities are also established.  Optimal analogous versions of (\ref{ONE}) and (\ref{TWO}) are established for $ p \neq 2$ which, in the case that $ \mu$ is a Dirac mass, answers  a best constant question posed by Adimurthi and Sekar (see [AS]).      

Since the above inequalities do not attain a standard question to ask is for what functions $ 0 \le V(x)$ do we have 
\begin{equation}
\label{THREE}
 \int_\Omega | \nabla u|_A^2dx \ge \frac{1}{4} \int_\Omega \frac{| \nabla E|_A^2}{E^2} u^2dx + \int_\Omega V(x) u^2dx, \qquad u \in H_0^1(\Omega).
 \end{equation}   Necessary and sufficient conditions on $V$ are obtained (in terms of the solvability of a linear pde) 
  for (\ref{THREE}) to hold.  Analogous results involving improvements are obtained for the weighted versions.  

 We establish optimal inequalities which are similar to (\ref{ONE}) and are valid for  $ u \in H^1(\Omega)$.      We obtain results on improvements of this inequality which are similar to the  above results on improvements.  In addition weighted versions of this inequality are also obtained.

 We remark that most of the known Hardy inequalities (in the case where $ p=2$) can be obtained, via the above approach, by making  suitable choices for $E$ and $A(x)$.

\end{abstract}

 \section{Introduction}   

 We begin by recalling the various Hardy inequalities.  Let $ \Omega $ be a bounded domain in $ \IR^n$  containing the origin and where $ n \ge 3$.  Then  Hardy's  inequality (see [OK]) asserts that 
 \begin{equation} \label{HS}
 \int_\Omega | \nabla u|^2 dx \ge \left( \frac{n-2}{2} \right)^2 \int_\Omega \frac{u^2}{|x|^2} dx, 
 \end{equation} for all $ u \in H_0^1(\Omega)$.  Moreover the constant $ \left( \frac{n-2}{2} \right)^2$ is optimal and not attained.   An analogous result asserts that for any bounded convex domain $ \Omega \subset \IR^n$ with smooth boundary and $ \delta(x):=dist(x,\pOm)$ (the euclidean distance from $x$ to $ \pOm$), there holds (see [BM]) 
 \begin{equation} \label{bound-HS} 
 \int_\Omega | \nabla u|^2 dx \ge \frac{1}{4} \int_\Omega \frac{u^2}{\delta^2} dx, 
 \end{equation} for all $ u \in H_0^1(\Omega)$.     
 Moreover the constant $ \frac{1}{4}$ is optimal and not attained.  We will refer to this inequality as Hardy's boundary inequality.    
 
    Recently Hardy inequalities involving more general distance functions than the distance to the origin or distance to the  boundary have been studied (see [BFT]).  Suppose $ \Omega$ is a domain in $ \IR^n$ and $M$ a piecewise smooth surface of co-dimension $k,$ $ k=1, ... , n$.  In case $ k=n$ we adopt the convention that $M$ is a point, say, the origin.  Set $ d(x):= dist(x,M)$.      Suppose $ k \neq 2$ and $ - \Delta d^{2-k} \ge 0 $ in $ \Omega \backslash M$  then 
 \begin{equation} \label{AAA}
 \int_\Omega | \nabla u|^2 dx \ge \frac{(k-2)^2}{4} \int_\Omega \frac{u^2}{d^2} dx, 
 \end{equation}
 for all $ u \in H_0^1( \Omega \backslash M)$.     We comment that the above inequalities all have $L^p$ analogs.
 
 In the last few years improved versions of the above inequalities have been obtained, in the sense that non-negative terms are added to the right hand sides of the inequalities; see [BV], [BM], [BFT], [BMS],[FT],[FHT],[VS].   One common type of improvement for the above Hardy inequalities are the so called potentials; we call $ 0 \le V(x)$, defined in $ \Omega $, a potential for (\ref{HS}) provided 
  \[ \int_\Omega | \nabla u|^2 dx - \left( \frac{n-2}{2} \right)^2 \int_\Omega \frac{u^2}{|x|^2} dx \ge \int_\Omega V(x) u^2 dx, \qquad u \in H_0^1(\Omega). \]   Most of the results in this direction are explicit examples of potentials $V$ where, in the best results, $V$ is an infinite series involving complicated inductively defined functions.   Very recently Ghoussoub and Moradifam [GM] gave the following necessary and sufficient conditions for a radial function $ V(x)=v(|x|)$ to be a potential in the case of  Hardy's inequality (\ref{HS}) on a radial domain $\Omega$: \\ $V$ is a potential if and only if there exists a positive function $y(r)$ which solves $ y''+ \frac{y'}{r} + vy=0 $ in $ (0, \sup_{x \in \Omega}|x|)$.
   
  In another direction people have considered Hardy  inequalities for operators more general than the Laplacian.   One case of this is the results obtained by Adimurthi and A. Sekar [AS]: \\     Suppose $ \Omega$ is a smooth domain in $ \IR^n$ which contains the origin, $ A(x) = ((a^{i,j}(x)))$ denotes a symmetric, uniformly positive definite matrix with suitably smooth coefficients and  for $ \xi \in \IR^n$ we define $ |\xi |_{A}^2:=| \xi |_{A(x)}^2:= A(x) \xi \cdot \xi $.  Now suppose $ E $ is a solution of $\L_{A,p}(E):= - div \left( | \nabla E|_A^{p-2} A \nabla E \right) = \delta_0  $ in $ \Omega$ with $ E=0$ on $ \pOm$ where $ \delta_0$ is the Dirac mass at $0$.    Then for all $ u \in W_0^{1,p}(\Omega)$ 
  \[ \int_\Omega | \nabla u|_A^2 dx- \left( \frac{p-1}{p} \right)^p \int_\Omega \frac{| \nabla E|_A^p}{E^p} |u|^pdx \ge 0. \] Improvements of this inequality were also obtained and they posed the following question: Is $ \left( \frac{p-1}{p} \right)^p $ optimal?\\    We show this is the case, even  for a more general inequality.    

 After completion of this work we noticed that various people had taken a similar approach to generalized Hardy inequalities, see [DL], [KMO],[LW].

\subsection{Outline and approach}

Our approach will be similar to the one taken by Adimurthi and A. Sekar but we mostly concentrate on the quadratic case ($p=2$) and for this we define  $ \L_A(E):=-div(A \nabla E)$.

 We now motivate our main inequality.  Suppose $ E $ is a smooth positive function defined in $ \Omega$. 
 Let $ u \in C_c^\infty(\Omega)$ and set $ v:=E^\frac{-1}{2} u$.    Then a calculation shows that 
  \[ | \nabla u|_A^2 - \frac{| \nabla E|_A^2}{4 E^2} u^2 = E | \nabla v|_A^2 + \frac{ A \nabla E \cdot \nabla (v^2)}{2}, \qquad \mbox{in $ \Omega$} \] and after integrating this over $ \Omega$ we obtain 
  \begin{equation} \label{start}
   \int_\Omega | \nabla u|_A^2 dx - \frac{1}{4} \int_\Omega \frac{| \nabla E|_A^2}{E^2} u^2 dx = \int_\Omega E | \nabla v|_A^2 dx + \frac{1}{2} \int_\Omega \frac{u^2}{E} \L_A(E) dx. 
   \end{equation}  If we further assume that $ \L_A(E) \ge 0 $ in $ \Omega$ then  
  \begin{equation} \label{genhardy}
   \int_\Omega | \nabla u|_A^2dx \ge \frac{1}{4} \int_\Omega \frac{| \nabla E|_A^2}{E^2}u^2dx, \qquad u \in H_0^1(\Omega).
   \end{equation}
    From this we see that the optimal constant $ C(E)$
  \[C(E):= \inf \left\{ \frac{ \int_\Omega | \nabla u|_A^2 }{  \int_\Omega \frac{| \nabla E|_A^2}{E^2}u^2 }dx : \; u \in H_0^1(\Omega) \backslash \{0\} \right\} \ge \frac{1}{4}. \]   It is possible to show that for all non-zero $u \in H_0^1(\Omega)$ we have 
  \[ \int_\Omega E | \nabla v|_A^2 dx > 0, \] where $ v$ is defined as above.  
  Using this and (\ref{start}) one sees that if $C(E) = \frac{1}{4}$ then $C(E)$ is not attained and hence if $C(E)$ is attained then $ C(E) > \frac{1}{4}$.   This shows that $ \frac{| \nabla E|_A^2}{E^2}$ needs to be singular if we want $C(E)=\frac{1}{4}$.   In fact one can show that   $ H_0^1(\Omega)$ is compactly embedded in $ L^2(\Omega, \frac{| \nabla E|_A^2}{E^2} dx) $ if $ \frac{ | \nabla E|_A^2}{E^2} \in L^p(\Omega)$ for some $ p > \frac{n}{2}$ and so one could then apply standard compactness arguments to show that $C(E)$ is attained.   We are  only interested in the case where $ C(E)= \frac{1}{4}$ and hence we need to ensure $ \frac{| \nabla E|_A^2}{E^2}$ is singular and this can be done in two obvious ways.  This naturally leads one to consider the following two classes of functions  $E$ (weights).

\begin{dfn} Suppose $0 <E$ in $ \Omega$ and $ \L_A(E) $ is a nonnegative nonzero finite measure in $ \Omega$ denoted by $ \mu$.   \\
1) \quad   If in addition $ E \in H_0^1(\Omega)$ then we call $E$ a boundary weight on $ \Omega$.   \\
2)  \quad If in addition $ E \in C^\infty(\Ov \backslash K)$ where $K \subset \Omega$ denotes the support of $ \mu$, $ E= \infty$ on $K$ and $dim_{box}(K)<n-2$ (see below) then we call $E$ an interior weight on $ \Omega$.
\end{dfn}

Given a compact subset $ K$ of $ \IR^n$ we define the box-counting dimension (entropy dimension) of $K$  by 
  \[ dim_{box}(K):= n - \lim_{r \searrow 0} \frac{ \log( \mathcal{H}^n(K_r))}{\log(r)} \] provided this limit exists and where $K_r:=\{ x \in \Omega: dist(x,K) < r \}$ and $ \mathcal{H}^\alpha $ is the $ \alpha $- dimensional Hausdorff measure.

\begin{remark}
 It is possible to show that $ C_c^{0,1}(\Omega \backslash K)$ 
 is dense in $W_0^{1,p}(\Omega)$ provided $K$ is compact and $ dim_{box}(K) < n-p$ (use appropriate Lipschitz cut off functions). 
 \end{remark}
 From here on $ \mu$ will  denote the measure $\L_A(E)$ and in the case where $ E$ is an interior weight on $ \Omega$, $K$ will denote the support of $ \mu$.   
 
 We now list the main results.  \\
 We show that if $E$ is either an interior or a  boundary weight in  $\Omega$ then we have the following inequality:
\begin{equation} \label{u}
\int_\Omega | \nabla u|_A^2dx  - \frac{1}{4} \int_\Omega \frac{ | \nabla E|_A^2}{E^2}u^2dx \ge 0, \qquad u \in H_0^1(\Omega) 
\end{equation}   with optimal constant which is not attained. \\  
 In the case that $E$ is a boundary weight on $ \Omega$  we  obtain 
\begin{equation} \label{uu}
\int_\Omega | \nabla u|_A^2dx - \frac{1}{4} \int_\Omega \frac{| \nabla E|_A^2}{E^2} u^2dx \ge \frac{1}{2} \int_\Omega \frac{u^2}{E} d \mu,  \qquad   u \in H_0^1(\Omega).
\end{equation}   Moreover $ \frac{1}{2}$ is optimal (once one fixes $ \frac{1}{4}$) and is not attained.

Using the methods developed in [GM] we  obtain necessary and sufficient conditions  on $ 0 \le V(x)$ to be a potential for (\ref{u}) in the case where $ E $ is an interior weight.   We show  that the following are equivalent: \\
1) \quad  For all $ u \in H_0^1(\Omega)$
\[ \qquad \int_\Omega | \nabla u|_A^2dx  - \frac{1}{4} \int_\Omega \frac{ | \nabla E|_A^2}{E^2}u^2dx \ge \int_\Omega V(x) u^2dx.\] 
2) \quad There exists some $ 0< \theta \in C^2(\Omega \backslash K) $ such that 
\begin{equation} 
   \frac{ -\L_A(\theta)}{\theta} + \frac{| \nabla E|_A^2}{4 E^2} + V \le 0 \qquad \mbox{ in $ \Omega \backslash K$}.
   \end{equation} 
If we further assume that $ E= \gamma \ge 0$ (constant) on $ \pOm$ and if we are only interested in potentials of the form $ V(x)=f(E) | \nabla E|_A^2$ then we can replace 2) with  \\
  2') \quad There exists some $ 0 < h \in C^2(\gamma, \infty)$ such that 
  \begin{equation}
  h''(t)+ \left( f(t)+ \frac{1}{4t^2} \right) h(t) \le 0, \qquad \mbox{ in $ ( \gamma, \infty)$.}
  \end{equation}   In practice this ode classification is more useful because of the shear abundance of ode results in the literature.  
   
We obtain weighted versions of (\ref{u}) (respectively (\ref{uu})) in the case that $E$ is an interior weight (respectively boundary weight) on $ \Omega$  which can be viewed as generalized versions of the Cafferelli-Kohn-Nirenberg inequality.  To be more precise we obtain: \\ 
Suppose $ E $ is an interior weight on $ \Omega$ and  $ t \neq \frac{1}{2}$.  Then 
  \begin{equation} \label{wwww}
  \int_\Omega E^{2t} | \nabla u|_A^2dx \ge ( t - \frac{1}{2} )^2 \int_\Omega | \nabla E|_A^2 E^{2t-2} u^2dx, 
  \end{equation} for all $ u \in C_c^{0,1}(\Omega \backslash K)$.  Moreover  the constant is optimal and not attained in the naturally induced function space.     \\
  Suppose $E$ is a boundary weight on $ \Omega$ and  $ 0 \neq t < \frac{1}{2}$.     Then (\ref{wwww}) holds for all $ u \in C_c^\infty(\Omega)$ and is not attained in the naturally induced function space.  Similarly 
   \begin{equation} 
   \int_\Omega E^{2t} | \nabla u|_A^2dx - ( t - \frac{1}{2})^2 \int_\Omega | \nabla E|_A^2 E^{2t-2} u^2dx \ge ( \frac{1}{2}-t) \int_\Omega E^{2t-1} u^2 d \mu, 
   \end{equation} for all $ u \in  C_c^\infty(\Omega)$.  Moreover the constant on the right is optimal and not attained in the natural function space.    In addition we show that the class of potentials for (\ref{wwww})   is given by $ \{ E^{2t} V: \mbox{$V$ is a potential for (\ref{u})} \}$. 
  
  We also examine generalized Hardy inequalities which are valid for functions $ u \in H^1(\Omega)$.    Suppose $ E $ a positive function with $ \L_A(E)+E $ a nonnegative nonzero finite measure denoted by $ \mu$, $ E = \infty$ on the $K$ (as before $K$ denotes the support of $\mu$)  and where we assume that $E$ satisfies a Neumann  boundary condition.  Then  
  \begin{equation} \label{non}
   \int_\Omega | \nabla u|_A^2 dx+ \frac{1}{2} \int_\Omega u^2 dx \ge \frac{1}{4} \int_\Omega \frac{| \nabla E|_A^2}{E^2}u^2dx, \qquad u \in H^1(\Omega).
   \end{equation}  Moreover these constants are optimal (in the sense that if one is fixed then the other is optimal).    Improvements of (\ref{non}) are also obtained.   Assuming the same conditions on $E$ we show that for $ 0 \le V$ we have 
   \[ \int_\Omega | \nabla u|_A^2 + \frac{1}{2} \int_\Omega u^2 dx - \frac{1}{4} \int_\Omega \frac{|\nabla E|_A^2}{E^2}u^2 dx \ge \int_\Omega V(x) u^2 dx, \qquad u \in H^1(\Omega)\] if and only if there exists some $ 0 < \theta \in C^\infty(\Ov \backslash K)$ such that 
   \begin{equation} \label{EQ2} -\L_A(\theta)- \frac{\theta}{2} + \frac{| \nabla E|_A^2}{4E^2} \theta +V \theta \le 0 \qquad \mbox{in $ \Omega \backslash K$},
   \end{equation} with $ A \nabla \theta \cdot \nu  =0$ on $ \pOm$.    \\
   Weighted version of (\ref{non}) are established.    Assuming the same conditions on $E$ we show that for $ t \neq \frac{1}{2}$ we have 
  \[ \int_\Omega E^{2t}| \nabla u|_A^2dx + \frac{1}{2} \int_\Omega E^{2t} u^2 dx \ge \left( t - \frac{1}{2} \right)^2 E^{2t-2} | \nabla E|_A^2 u^2 dx,\] for all $ u \in C_c^\infty(\Ov \backslash K)$.  Moreover the constants are optimal and not obtained in the naturally induced function space.  
  
  We establish optimal Hardy inequalities which are valid on exterior and annular domains.   Suppose $ \Omega $ is a exterior domain in $ \IR^n$, $ E>0$ in $ \IR^n$, $ \lim_{|x| \rightarrow \infty} E=0$, $ -\Delta E = \mu$ in $ \IR^n$ where $ \mu$ is a nonzero nonnegative finite measure with compact support $K$.  In addition we assume that $ dist(K,\Omega)>0$ and $ \partial_\nu E \ge 0 $ on $ \pOm$.  We define  $ D^1(\Omega \cup \pOm) $ to be the completion of $ C_c^\infty(\Omega \cup \pOm)$ with respect to the norm $ \| \nabla u\|_{L^2(\Omega)}$. Then  \\
  \noindent (i) \quad For all $ u \in D^1(\Omega \cup \pOm)$ we have 
 \begin{equation}  \label{Ext200}
 \int_\Omega | \nabla u|^2 dx \ge \frac{1}{4} \int_\Omega \frac{ | \nabla E|^2}{E^2} u^2 dx.
 \end{equation}  Moreover the constant is optimal and not attained.   \\
 (ii) \quad For all $ u \in D^1(\Omega \cup \pOm)$ we have 
 \begin{equation} 
 \int_\Omega | \nabla u|^2dx \ge \frac{1}{4} \int_\Omega \frac{ | \nabla E|^2}{E^2} u^2 dx + \frac{1}{2} \int_{\pOm} \frac{u^2 \partial_\nu E}{E} dS(x).
 \end{equation}
 
 Now suppose $ \Omega=\Omega_2 \backslash \overline{\Omega_1}$ where $ \Omega_1 \subset \subset \Omega_2$ are both connected and $ \Omega $ is connected.  Suppose  $ 0<E $ in $ \Omega_2$ and $ -\Delta E = \mu$ in $ \Omega_2$ where $ \mu$ is a nonnegative nonzero finite measure compactly supported in $ \Omega_1$.   In addition we assume that $ E=0$ on $ \pOm_2$ and $\partial_\nu E \le 0 $ on $ \pOm_1$.  Then (\ref{Ext200}) is optimal and not attained over $ H_0^1(\Omega \cup \pOm_1):=\{ u \in H^1(\Omega): u=0 \; \mbox{ on } \; \pOm_2\}$.

  Optimal non-quadratic Hardy inequalities are also obtained in both the interior and boundary  cases.

  \subsection{Examples}
  
  We now look at various examples of Hardy inequalities (and applications of) which can be obtained  after making suitable choices of weights $E$ and matrices $A$.   In most of the examples we will take $ A $ to be the identity matrix.

  \begin{enumerate}
 \item \textbf{Hardy's inequality:} \quad Let $ \Omega $ denote a domain in $ \IR^n$ ($n\ge3$) which contains the origin and set $ E(x):=|x|^{2-n}$.  Then $ -\Delta E = c \delta_0$ where $ c >0$ and $ \delta_0$ is the Dirac mass at $0$.  Also $ \frac{ | \nabla E|^2}{4 E^2}=  \left( \frac{n-2}{2} \right)^2 \frac{1}{|x|^2} $ and so (\ref{u}) gives the  Hardy's inequality.

\item \textbf{Hardy's inequality in dimension two:} \quad Now suppose $ \Omega $ is a domain in $ \IR^2$ which contains the origin.  Put  $ E(x):= -\log(R^{-1} |x|) $ where  $ R:= \sup_\Omega |x|$.   Then $ -\Delta E = c \delta_0$ where $ c >0$ and putting $ E $ into (\ref{u}) gives 
 \[ \int_\Omega | \nabla u|^2 dx \ge \frac{1}{4} \int_\Omega \frac{u^2}{|x|^2 \log^2(R^{-1} |x|) }dx, \qquad u \in C_c^\infty(\Omega). \]
 
 \item \textbf{Hardy's boundary  inequality:} \quad Let $ \Omega $ denote a bounded convex set in $ \IR^n$ and set $ E(x):=\delta(x):=dist(x,\pOm)$.  Since $ \Omega $ is convex one can show $ \delta $ is concave and hence $ - \Delta \delta \ge 0$ in $ \Omega$.   Putting  $ E $ into (\ref{uu}) gives an improved version of (\ref{bound-HS}).
 
 \item \textbf{Hardy's boundary inequality in the unit ball:}  \quad Let $B$ denote the unit ball in $ \IR^n$ and set $ E(x):=1-|x|$.  Putting $E$ into (\ref{uu}) gives 
 \[ \int_B | \nabla u|^2dx \ge \frac{1}{4} \int_B \frac{u^2}{(1-|x|)^2}dx + \frac{n-1}{2} \int_B \frac{u^2}{|x|(1-|x|)}dx, \qquad u \in C_c^\infty(B).\]
 
 \item \textbf{Intermediate case:} \quad Set $ E(x):=d(x)^{2-k} $ where $ d$ and $k$ are as in (\ref{AAA}).  Since $ -\Delta E \ge 0 $ we obtain (\ref{AAA}) after subbing $E$ into (\ref{u}). 
 
  \item \textbf{Hardy's boundary inequality in the half space:} \quad Let $ \IR^n_+$ denote the half space and set $ E(x):=dist(x, \IR^n_+)=x_n$.  Then putting $ E $ into (\ref{u}) gives 
 \[ \int_{\IR^n_+} | \nabla u|^2dx \ge \frac{1}{4} \int_{\IR^n_+} \frac{u^2}{x_n^2}dx, \qquad u \in C_c^\infty(\IR^n_+).\]  Maz'ja (see [M]) obtained the following improvement
 \[\int_{\IR^n_+} | \nabla u|^2dx - \frac{1}{4} \int_{\IR^n_+} \frac{u^2}{x_n^2}dx \ge \frac{1}{16} \int_{\IR^n_+} \frac{u^2}{ (x_n^2 + x_{n-1}^2)^\frac{1}{2} x_n}dx , \qquad u \in C_c^\infty(\IR^n_+). \]  One might ask whether we can take a more symmetrical potential in the improvement, say something like $V(x)=f(x_n)$ where $ f $ is strictly positive.  Using our ode classification of potentials we will see that this is not possible.

 \item \textbf{Hardy's   inequality valid for $ u \in H^1(\Omega)$:} \quad Let $ B$ denote the unit ball in $ \IR^3$ and set $ E(x):=|x|^{-1}e^{|x|}$.   Then a computation shows that 
 \[ -\Delta E + E = 4 \pi^2 \delta_0 \qquad \mbox{in $B$} \] and where $ \partial_\nu E =0 $ on $ \partial B$. Here $ \delta_0$ is the Dirac mass at $0$.   Putting $E$ into (\ref{non}) we see that 
 \[ \int_B | \nabla u |^2 dx+ \frac{1}{2} \int_B u^2 dx\ge \frac{1}{4} \int_B \frac{ (1-|x|)^2}{|x|^2} u^2dx, \qquad u \in H^1(B).\]   Also the constants are optimal (in the sense mentioned in (\ref{non})) and are not attained.  
 
 \item \textbf{$H^1(\Omega)$ Hardy inequalities in exterior domains:} \quad Let  $ \Omega$ denote an exterior domain in $ \IR^n$ with $ n \ge 3$, $ 0 \notin \overline{\Omega}$ and such that $ \nu(x)\cdot x \le 0$ for all $ x \in \pOm$.  Setting $ E:=|x|^{2-n}$ in (\ref{Ext200}) we obtain 
\[ \int_{\Omega} | \nabla u|^2 dx \ge \left( \frac{n-2}{2} \right)^2 \int_{\Omega} \frac{u^2}{|x|^2} dx, \] for all $ u \in C_c^\infty(\Omega \cup \pOm)$.  Moreover the constant is optimal and not attained in the naturally induced function space.

 \item \textbf{Hardy's inequality in a annular domain:} \quad  Assume that $ 0 \in \Omega_1 \subset \subset B_R \subset \IR^2$ where $ \Omega_1$ is connected and $ B_R$ is the open ball centered at $0$ with radius $R$.  In addition we assume that $ x \cdot \nu(x) \ge 0$ on $ \pOm_1$ where $ \nu$ is the outward pointing normal.  Define $ \Omega:=B_R \backslash \overline{\Omega_1}$, which we assume is connected,  and set $ E(x):=-\log( R^{-1}|x|)$.  Then by the above mentioned results on annular domains one has 
 \[ \int_{\Omega} | \nabla u|^2 dx \ge \frac{1}{4} \int_\Omega \frac{u^2}{|x|^2 \log^2(R^{-1}|x|)} dx, \]
 for all $ u \in H_0^1(\Omega \cup \Omega_1)$.  Moreover the constant is optimal and not attained.

 \item  Suppose $ E >0 $  in $ \Omega$,  let $ f:(0,\infty) \rightarrow (0,\infty)$ and set $ \tilde{E}:=f(E)$.  Putting $ \tilde{E} $ into (\ref{start}) for $E$ gives 
 \[ \int_\Omega | \nabla u|_A^2 dx \ge \int_\Omega | \nabla E|_A^2 \left( \frac{ f'(E)^2}{4 f(E)^2} - \frac{ f''(E)}{2 f(E)} \right) u^2dx + \frac{1}{2} \int_\Omega \frac{f'(E) \L_A(E)}{f(E)} u^2dx, \qquad u \in C_c^\infty(\Omega).\]   An important example will be when $ f(E):=E^t$ where $ 0 < t <1$;  in fact we will use $ E(x):=\delta(x)^t$ ($\delta(x):=dist(x,\pOm)$) to show that if one drops the requirement that $ \mu$ is a \emph{finite} measure (and just assumes $ \mu $ a  locally finite measure) (\ref{u}) need not be optimal.

 \item \textbf{Eigenvalue bound:} \quad Let $ \Omega $ be a bounded subset of $ \IR^n$ and $E>0$, $ \L_A(E) \ge 0 $ in $ \Omega$ with $ | \nabla E|_A^2 = 1$ a.e. in $ \Omega$.     Let $ \lambda_A(\Omega)$ denote the first eigenvalue of $\L_A$ in $ H_0^1(\Omega)$.  Then $ \lambda_A(\Omega) \| E \|_{L^\infty}^2 \ge \frac{\pi^2}{4}$.
   To show this one puts $ f(z):=\sin^2( \frac{ \pi z}{2 \| E \|_{L^\infty}}) $ into the above result and drops the term involving the measure.

\item   Suppose $ E $ is an interior weight on $ \Omega$ with $ E =1$ on $ \pOm$.  Then by using the above result with $ f(E):=(\log(E))^\frac{1}{2}$ one obtains the  inequality
\[ \int_\Omega | \nabla u|_A^2dx \ge \frac{1}{16} \int_\Omega \frac{ 3+4 \log(E)}{E^2 \log^2(E)} | \nabla E|_A^2 u^2dx, \qquad u \in H_0^1(\Omega).\]  Taking instead  $ f(E):=E \log(E)$  gives 
\[ \int_\Omega | \nabla u|_A^2 dx \ge \frac{1}{4} \int_\Omega \frac{\log^2(E)+1}{E^2 \log^2(E)} u^2dx, \qquad u \in H_0^1(\Omega).\]

   \item \textbf{Poincare's inequality in an unbounded slab:} In general 
   \[ \int_\Omega | \nabla u|^2 dx\ge C \int_\Omega u^2dx, \qquad u \in C_c^\infty(\Omega)\] does not hold for unbounded domains.   It is known that for certain unbounded domains the inequality does in fact hold.  One example would be $ \Omega:=\{ x \in \IR^n: 0 <x_n < \pi \}$.  We now use (\ref{start}) to show a slightly stronger result.  Put $ E(x):=\sin(x_n)$ into (\ref{start}) and drop a term to arrive at 
   \[ \int_\Omega | \nabla u(x)|^2 dx \ge \frac{1}{4} \int_\Omega \frac{u(x)^2}{ \tan^2(x_n)} dx + \frac{1}{2} \int_\Omega u(x)^2 dx, \qquad u \in C_c^\infty(\Omega).\]

   \item \textbf{Hardy's boundary inequality in a cone:} Put $ \Omega :=(0,\infty) \times (0,\infty)$ and $E(x):=dist(x,\pOm)=\min\{ x_1,x_2\}$.  Then $ -\Delta E = \sqrt{2} \sigma$ where $ \sigma $ is the measure associated with the line $ \Gamma:=\{x:  x_2=x_1\}$.   Putting $E$ into (\ref{start}) gives 
   \[ \int_\Omega | \nabla u|^2dx \ge \frac{1}{4} \int_\Omega \frac{u^2}{( \min\{x_1,x_2\})^2}dx + \frac{1}{\sqrt{2}} \int_\Gamma \frac{u^2}{ \min\{x_1,x_2\} } d \sigma, \qquad u\in C_c^\infty(\Omega). \]

 \item  Suppose $ -\Delta \phi =1 $ in $ \Omega $ with $ \phi =0 $ on $ \pOm$.  Define $ E:=e^{t \phi}-1$.  Then $ -\Delta E = t e^{t \phi}( 1 - t | \nabla \phi|^2) $ which is non-negative for  sufficiently small $t>0$.  Then $E$ is a boundary weight and hence putting $E$ into (\ref{uu}) gives 
 \[ \int_\Omega | \nabla u|^2 dx\ge \frac{t^2}{4} \int_\Omega \frac{ e^{2t \phi} | \nabla \phi|^2}{(e^{t\phi}-1)^2} u^2dx + \frac{t}{2} \int_\Omega \frac{e^{t \phi} ( 1 - t | \nabla \phi|^2)}{e^{t \phi}-1} u^2dx, \qquad u \in H_0^1(\Omega) \] which is optimal.  Sending $ t \searrow 0$ recovers (\ref{uu}) with $E=\phi$.
 
 \item \textbf{Trace theorem:} \quad  Let $ \Omega $ denote a domain in $ \IR^n$ where $ n \ge 3 $ and such that $ B \subset \subset \Omega$ (here $ B$ is the unit ball).   Define 
 \begin{equation*}
 E(x):= \left\{
 \begin{array}{ll}
 1 & \qquad |x| < 1 \\
 \frac{1}{|x|^{n-2}} & \qquad  |x| >1.
 \end{array}
 \right.
 \end{equation*}   A computation shows that $ -\Delta E = c \sigma $ where $ c >0$ and where $ \sigma $ is the surface measure associated with $ \partial B$.   Putting this $ E $ into (\ref{uu}) and dropping a couple of terms gives 
 \[ \int_\Omega | \nabla u|^2 dx \ge \frac{c}{2} \int_{\partial B} u^2 d \sigma, \qquad u \in C_c^\infty(\Omega).\]    
 
 \item \textbf{Regularity:}  Suppose $ E \in L^\infty_{loc}(\Omega) $ is a positive solution to $ \L_A(E) = \mu $ in $ \Omega $ where $ \mu$ is locally finite measure.   Then using (\ref{u}) we see that $ E \in H^1_{loc}(\Omega)$.
 
 \item  \textbf{Baouendi-Grushin operator:}  Here we mention that various operators can be put into the form we are interested in.  Suppose $ \Omega $ is an open subset of $ \IR^N =\IR^n \times \IR^k$ and $ \xi \in \Omega $ is written $ \xi=(x,y) $ using the above decomposition of $ \IR^N$.    For $ \gamma >0$ one defines the vector field $ \nabla_\gamma:=( \nabla_x, |x|^\gamma \nabla_y)$ and the Baouendi-Grushin operator $ {\L_A}:= -\Delta_x - |x|^{2 \gamma} \Delta_y$.   Take 
\begin{displaymath}
A(\xi):= \left( \begin{array}{clc} 
I_n & 0 \\
0 & |x|^{2 \gamma} I_k 
\end{array} \right)
\end{displaymath} where $ I_n, I_k$ are the identity matrices of size $n$ and $k$.   Then $ | \nabla_\gamma E|^2=| \nabla E|_A^2$ and $ -div(A \nabla E)= {\L_A}(E)$.  

 \end{enumerate}

\section{Main Results} 

Throughout this article we shall assume that $ \Omega$ is a bounded connected domain in $ \IR^n$ (unless otherwise mentioned) with smooth boundary
 and $A(x) = ((a^{i,j}(x)))$ is a $ n \times n$    symmetric, uniformly positive definite matrix with $ a^{i,j} \in C^\infty(\Ov)$ and for $  \xi \in \IR^n$ we define $ |\xi |_{A}^2:=| \xi |_{A(x)}^2:= A(x) \xi \cdot \xi $.  
 
 If $E$ is an interior weight or a boundary weight on $ \Omega$  we have, by the strong maximum principle (see [V]), $E$  bounded away from zero on compact subsets of $ \Omega$. 
  
The following theorem gives the main inequalities.  In addition we consider a slight generalization of the case where $E$ is a boundary weight on $ \Omega$. 

\begin{thm} \label{first}  
\flushleft 
(i) \quad  Suppose $ E $ is either an interior or a  boundary weight on  $\Omega$.  Then
 
\begin{equation} \label{hardy}
\int_\Omega | \nabla u|_A^2dx  - \frac{1}{4} \int_\Omega \frac{ | \nabla E|_A^2}{E^2}u^2dx \ge 0, 
\end{equation} for all $ u \in H_0^1(\Omega)$.  Moreover $ \frac{1}{4}$ is optimal and not attained.  \\
(ii) \quad   Suppose $ E $ is a boundary weight on  $\Omega$.  Then  
\begin{equation} \label{extra}
\int_\Omega | \nabla u|_A^2dx - \frac{1}{4} \int_\Omega \frac{| \nabla E|_A^2}{E^2} u^2 dx\ge \frac{1}{2} \int_\Omega \frac{u^2}{E} d \mu, 
\end{equation} for all $ u \in H_0^1(\Omega)$.  Moreover $ \frac{1}{2}$ is optimal (once one fixes $ \frac{1}{4}$) and is not attained. \\
(iii) \quad Suppose $ E \in C^\infty(\Ov)$ with $ E>0$, $ \L_A(E) \ge 0$ in $ \Omega$ and   $ \Gamma:=\{x \in \pOm: E(x)=0\}$ contains $ B(x_0,r) \cap \pOm$ for some $ x_0 \in \pOm$ and $ r>0$.   Then (\ref{hardy}) is optimal.  

\end{thm}

\begin{remark} One can consider more general functions $E$. Most of the results (including the above one) concerning interior weights on $ \Omega$ can be generalized to the case where  $ \L_A(E) = \mu +h $, here $ \mu$ is again a nonnegative nonzero finite measure and $h$ is a suitably smooth non-negative function.   
\end{remark}

We begin by justifying (\ref{start}).

\begin{lemma} \label{startingpoint} (i) \quad Suppose $ E $ is an interior weight on $ \Omega$.  Then 
\begin{equation} \label{spint}
\int_\Omega | \nabla u|_A^2dx - \frac{1}{4} \int_\Omega \frac{ | \nabla E |_A^2}{E^2} u^2dx \ge \int_\Omega E | \nabla v|_A^2dx, 
\end{equation}
for all $ u \in C_c^{0,1}(\Omega \backslash K)$ and where $ v:=E^\frac{-1}{2} u$. \\
(ii) \quad Suppose $ E $ is a boundary weight on $ \Omega$.  Then 
\begin{equation} \label{spbound}
\int_\Omega | \nabla u|_A^2dx - \frac{1}{4} \int_\Omega \frac{ | \nabla E |_A^2}{E^2} u^2dx \ge \int_\Omega E | \nabla v|_A^2dx + \frac{1}{2} \int_\Omega \frac{u^2}{E} d \mu, 
\end{equation} for all $ u \in H_0^1(\Omega) $ and $ v:=E^\frac{-1}{2}u$.

\end{lemma}

\begin{proof} (i) \quad Since $E$ is smooth away from $K$ and noting the supports of both $u$ and $ v$ the integration by parts used in obtaining (\ref{start}) is valid.  \\ 
\noindent (ii) \quad Now suppose $ E $ in a boundary weight.  Extend $E$ to all of $ \IR^n$ by setting $E =0 $ outside of $ \Ov$ and let $ E_\E$ denote the $\E $ mollification of $E$.   Let $ u \in C_c^\infty(\Omega)$, $ v_\E:=E_\E^\frac{-1}{2} u$ and define $ F_\E:=\L_A(E_\E)$.    Now one easily obtains (\ref{start}) but with $ E $ and $ v$ replaced with $ E_\E, v_\E$.  Standard arguments show that $ u E_\E^{-1} \rightarrow u E^{-1}$ in $H_0^1(\Omega)$, $ | \nabla E_\E|_A^2 E_\E^{-2} \rightarrow | \nabla E|_A^2 E^{-2},  E_\E | \nabla v_\E|_A^2 \rightarrow E | \nabla v|_A^2 $ a.e. in $ \Omega$ and $ u F_\E \rightharpoonup u \mu  $ in $H^{-1}(\Omega)$.  Using these results along with Fatou's lemma allows us to pass to the limit.

\end{proof} 

\begin{remark} When we prove our various Hardy inequalities, which all stem from (\ref{start}) we will generally drop the term 
\[ \int_\Omega E \big| \nabla \left( \frac{u}{\sqrt{E}} \right)\big|_A^2 dx.\]   To show the given inequality does not attain we will generally just not drop this term.   This term is positive for non-zero $u$ provided $ u $ is not a multiple of $ \sqrt{E}$.   Since $ \sqrt{E} \notin H_0^1(\Omega)$ this will not be an issue.    In theorem \ref{alpha} this will be a concern.  

\end{remark}

As usual we will need an ample supply of test functions for best constant calculations.  The next lemma provides this. 
When $ E $ is an interior weight we let $ g $ denote a solution to $ \L_A(g)=0$ in $ \Omega$ with $ g=E$ on $ \pOm$. 

\begin{lemma} \label{test} Suppose $ E $ is an interior weight on $ \Omega$ and  $ 0 < \gamma:=\min_{\pOm} E$.  Then  \\
(i) \quad $ u_t:=E^t - g^t \in H_0^1(\Omega)$ for $ 0 < t < \frac{1}{2}$. \\
(ii) \qquad Define $ I(t):=\int_\Omega | \nabla E|_A^2 E^{2t-2}dx$.  Then $ I(t) $ is finite for $ t < \frac{1}{2}$ and $ I(t) \rightarrow \infty $ as $ t \nearrow \frac{1}{2}$.  \\
(ii) \quad Suppose $ E=\gamma >0 $ on $\pOm$.   Define $ v_{t,\tau}:=E^t \log^\tau( \gamma^{-1}E)$ and 
\[ J_t (\tau):= \int_\Omega E^{2t-2} | \nabla E|_A^2 \log^{2\tau-2}(\gamma^{-1}E) dx.\] 
Then $ v_{t,\tau} \in H_0^1(\Omega)$ for $ 0 < t < \frac{1}{2}$ and $ \tau > \frac{1}{2}$.  Moreover for each $ 0 < t  < \frac{1}{2}$ we have $ J_t(\tau) \rightarrow \infty$ as $ \tau \searrow \frac{1}{2}$.

\end{lemma}

\begin{proof} We prove the results up to some unjustified integration by parts; which can be justified by regularizing the measure, integrating by parts and passing to limits.   \\
\noindent (i), (ii) \quad  Fix $ 0 < t < \frac{1}{2}$ and then note that $ | \nabla u_t|^2 \le C E^{2t-2} | \nabla E|_A^2 + C g^{2t-2} | \nabla g|_A^2 $ where $C$ is some uniform constant.  The term involving $ g $ is  harmless.  Now multiply $ \L_A(E) = \mu $ by $ E^{2t-1}$ and integrate over $ \Omega $ to obtain
\begin{eqnarray*}
(1-2t) \int_\Omega E^{2t-2} | \nabla E|_A^2dx&=& - \int_{\pOm} g^{2t-1} (A \nabla E) \cdot \nu  \\
&=& \E(t) - \int_{\pOm} ( A \nabla E) \cdot \nu  \\
&=& \E(t) - \int_\Omega div(A \nabla E)dx \\
&=& \E(t) + \mu(\Omega),
\end{eqnarray*} where $ \E(t) \rightarrow 0 $ as $ t \nearrow \frac{1}{2}$.   Note $ \int_\Omega E^{2t-1} d \mu =0 $ since $ t < \frac{1}{2}$ and $ E = \infty $ on $K$.    From this we see that $I(t)= \int_\Omega | \nabla E|_A^2 E^{2t-2}dx< \infty$ and so $ u_t \in H_0^1(\Omega)$.   We also see that $ \lim_{t \nearrow \frac{1}{2}}I(t)=\infty$. \\
\noindent (iii) \quad Take $ 0 < t < \frac{1}{2}$, $ \tau > \frac{1}{2}$ and $ v_{t,\tau}$ defined as above.  One easily sees that $ v_{t, \tau}$ is continuous near $ \pOm$ and vanishes on $ \pOm$.  So to show $ v_{t,\tau } \in H_0^1(\Omega)$ it is sufficient to show 
\[  w_1:=E^{2t-2} | \nabla E|^2 \log^{2\tau} (\gamma^{-1}E), \; w_2:= E^{2t-2} | \nabla E|^2 \log^{2\tau-2}(\gamma^{-1}E) \in L^1(\Omega).\]  These functions are only singular near $ K $ and $ \pOm$.  Now set $ W_\tau:= E^{2t-2} | \nabla E|^2 \log^{2 \tau-2}(\gamma^{-1}E) $ and so $ w_2 = W_\tau$ and $ w_1=W_{\tau+1}$.  Now suppose $ t' \in (t,\frac{1}{2})$ and so 
\[ W_{\tau+1} =E^{2 t'-2} | \nabla E|^2 \frac{ \log^{2 \tau}(\gamma^{-1} E)}{ E^{2t'-2t}} \le C E^{2t'-2} | \nabla E|^2 \qquad \mbox{ near $K$ }, \] and so $ w_1 = W_{\tau+1} \in L^1(K_\E)$ where $ K_\E$ is a small neighborhood of $K$.  Now note that $ w_2 $ is better behaved than $ w_1 $ near $K$ and so we also have $ w_2 \in L^1(K_\E)$. \\ 
Define $ \Omega_\E:= \{ x \in \Omega: E(x) < \gamma +\E \}$ and take $ \E >0$ sufficiently small such that $ K \subset \Omega \backslash \Omega_{2 \E}$.  Now using the co-area formula we have 
\begin{eqnarray*}
\int_{\Omega_\E} E^{2t-2} | \nabla E|^2 \log^{2\tau-2}(\gamma^{-1}E) dx& \le & \sup_{\Omega_\E} | \nabla E| \int_{\Omega_\E} E^{2t-2} \log^{2\tau-2} ( \gamma^{-1} E) | \nabla E|dx  \\
& \le & C \int_1^{1+ \frac{\E}{\gamma}} s^{2t-2} \log^{2 \tau-2}(s) ds, 
\end{eqnarray*} which is finite for $ \tau > \frac{1}{2}$.    So we see that $ w_2 \in L^1(\Omega_\E)$ for sufficiently small $ \E>0$ and noting that $ w_1 $ is better behaved near $ \pOm$ than $ w_2$ we have the same for $ w_1$.   Combining these results we see that $ v_{t,\tau} \in H_0^1(\Omega)$. \\ Fix $ 0 < t < \frac{1}{2}$ and $ \tau > \frac{1}{2}$.  By Hopf's lemma we have $ | \nabla E(x) | $ bounded away from zero on $ \Omega_\E$ for $ \E >0$ sufficiently small; fix $ \E>0$ sufficiently small.   Then 
\begin{eqnarray*}
J_t(\tau) & \ge & C \int_{ \Omega_\E} E^{2t-2} \log^{2\tau-2}(\gamma^{-1}E) | \nabla E|dx \\ 
& \ge & \tilde{C} \int_1^{1 + \frac{\E}{\gamma}} s^{2t-2} \log^{2\tau-2}(s) ds, 
\end{eqnarray*}  and a computation shows the last integral becomes unbounded as $ \tau \searrow \frac{1}{2}$.

\end{proof}

 \noindent \textbf{Proof of theorem \ref{first}:}   
(i) \quad Using lemma \ref{startingpoint} and, in the case where $E$ is a interior weight on $ \Omega$, the fact that $ C_c^{0,1}(\Omega \backslash K)$ is dense in $H_0^1(\Omega)$ we obtain (\ref{hardy}).  We now show the constant is optimal.   Suppose $ E $ is an interior weight on $ \Omega$ and define $ E_\E:=\E+E$, $ g_\E:=\E +g$ where $ \E>0$.  Define $ I_\E(t):= \int_\Omega | \nabla E_\E|_A^2 E_\E^{2t-2}dx$.  As in the proof of lemma \ref{test} one can show that for each $ \E >0$ $ \lim_{t \nearrow \frac{1}{2}} I_\E(t) = \infty$.   We use $ u_{t,\E}:=E_\E^t - g_\E^t$ as test functions.  Let $ 0 < t < \frac{1}{2}$ and $ \E>0$.  Then 
\begin{eqnarray*}
Q_{t,\E}:= \frac{\int_\Omega  | \nabla u_{t, \E} |_A^2dx}{  \int_\Omega \frac{ | \nabla E_\E|_A^2}{E_\E^2} u_{t,\E}^2dx} & \le &  \frac{t^2 I_\E(t) + C_0 +C_1 \sqrt{I_\E(t)} }{ I_\E(t) - C_2 I_\E( \frac{t}{2}) - C_3 I_\E(0) }, \\
\end{eqnarray*} where the constants $ C_k$ possibly depend on $ \E$.     From this we see that $ \lim_{t \nearrow \frac{1}{2}} Q_{t,\E} = \frac{1}{4}$ after recalling $ Q_{t,\E} \ge \frac{1}{4}$.     Now fix $ \E>0$ and let $ u \in C_c^\infty(\Omega)$ be non-zero.   Then a simple computation shows 
\[ \frac{ \int_\Omega | \nabla u|_A^2dx }{ \int_\Omega \frac{| \nabla E|_A^2}{E^2} u^2dx} \le \frac{ \int_\Omega | \nabla u|_A^2dx }{ \int_\Omega \frac{| \nabla E_\E|_A^2}{E_\E^2} u^2dx}, \] which, when combined with the above facts, gives the desired best constant result.    To see $ \frac{1}{4}$ is not attained use (\ref{spint}).  \\
Now suppose  $ E $ is a boundary weight on $ \Omega$, $ \E>0$ and $ t > \frac{1}{2}$.    Define $ f_\E(z):= z^{2t-1}-\E^{2t-1}$ for $ z >\E$ and $0$ otherwise.    Using $ f_\E(E) \in H_0^1(\Omega)$ as a test function in the pde associated with $E$ one obtains, after sending $ \E \searrow 0$,
\begin{equation} \label{hhhhh}
 (2t-1) \int_\Omega E^{2t-2} | \nabla E|_A^2dx = \int_\Omega E^{2t-1} d \mu, 
 \end{equation}
  which shows that $ E^t \in H_0^1(\Omega)$ for $ \frac{1}{2} < t \le 1$.  To see $ \frac{1}{4}$ is optimal in (\ref{hardy}) use $ E^t $ (as $ t \searrow \frac{1}{2}$) as a minimizing sequence.  \\ 
\noindent (ii) \quad  Suppose $ E $ is a boundary weight on $ \Omega$.   Let $ \frac{1}{2} < t < 1 $ and so $ E^t \in H_0^1(\Omega)$.  Using (\ref{hhhhh}) we have 
\[ \frac{\int_\Omega | \nabla E^t|_A^2dx - \frac{1}{4} \int_\Omega \frac{| \nabla E|_A^2}{E^2} (E^t)^2}{ \int_\Omega \frac{(E^t)^2}{E} d \mu } = \frac{t}{2} + \frac{1}{4}, \] which shows that $ \frac{1}{2}$ is optimal.  \\
\noindent (iii) \quad Suppose $ E$  is as in the hypothesis.  The only issue is whether $ \frac{1}{4}$ is optimal. Without loss of generality assume that $ 0 \in \pOm$ and $ B(0,2R) \cap \pOm \subset \Gamma$.     Suppose $ 0 < r <R $ and define  
\begin{equation*}
\phi(x):= \left\{ \begin{array}{ll}1 & \qquad x \in \Omega(r) \\
\frac{R-|x|}{R-r} & \qquad x \in \Omega(R) \backslash \Omega(r) \\
0 & \qquad x \in \Omega \backslash \Omega(R),
\end{array}
\right.
\end{equation*} where $ \Omega(r):= B(0,r) \cap \Omega$.  Define $ u_t:=E^t \phi$ which can be shown to be an element of $H_0^1(\Omega)$ for $  t > \frac{1}{2}$.  One uses $ u_t $ as $ t \searrow \frac{1}{2}$ as a minimizing sequence along with arguments similar to the above to show $ \frac{1}{4}$ is optimal.

\hfill $ \Box$

 The following example shows that if we just assume that $ 0 < E \in H_0^1(\Omega)$ with $ \L_A(E)$ a locally finite measure then (\ref{hardy}) need not be optimal.
 \begin{exam} Take $ \Omega $ a bounded convex domain in $ \IR^n$ and set $ \delta(x):=dist(x,\pOm)$.  Fix $ \frac{1}{2} < t < 1 $ and set $ E:=\delta^t \in H_0^1(\Omega)$.    Then 
 \[ \frac{| \nabla E|^2}{E^2} = \frac{t^2}{\delta^2}, \qquad \mu:=-\Delta E = t(1-t) \delta^{t-2} + t \delta^{t-1} (-\Delta \delta) \ge 0 \qquad \mbox{ in $ \Omega$}, \] and so putting $E$ into (\ref{hardy}) gives 
 \[ \int_\Omega | \nabla u|^2dx \ge \frac{t^2}{4} \int_\Omega \frac{u^2}{\delta^2} dx, \] for $ u \in H_0^1(\Omega)$.   This shows that (\ref{hardy}) was not optimal.  This apparent failure of theorem \ref{first} is due to the fact $ \mu$ not a finite measure; use the co-area formula to show $ \delta^{t-2} \notin L^1(\Omega)$. 
 
 \end{exam}

 We now give an alternate way to view best constants in (\ref{extra}).   Define $ \mathcal{C} $ to be the set of $ (\beta, \alpha) \in \IR^2$ such that 
 \begin{equation} \label{alphabeta}
 \int_\Omega | \nabla u|_A^2dx \ge \alpha \int_\Omega \frac{| \nabla E|_A^2}{E^2} u^2dx + \beta \int_\Omega \frac{u^2}{E} d \mu, \qquad    u \in H_0^1(\Omega). 
 \end{equation}

 \begin{thm} \label{alpha}  Suppose $ E \in L^\infty(\Omega)$ is a boundary weight on  $\Omega$.  Then 
 \[ \mathcal{C}= \left\{ ( \beta , \alpha): \beta > \frac{1}{2}, \alpha \le \beta - \beta^2 \right\} \cup \left( - \infty, \frac{1}{2} \right] \times \left( - \infty, \frac{1}{4} \right]=:\mathcal{C}'. \]  Moreover  (\ref{alphabeta})  does attain on $ \Gamma:=\{ (\tau, \tau-\tau^2): \tau> \frac{1}{2} \} \subset \partial \mathcal{C}$ and does not attain on $ \partial \mathcal{C} \backslash \Gamma$.

 \end{thm}

 \begin{proof} Using similar arguments to the above one can show that $ E^t \in H_0^1(\Omega)$ for all $ t > \frac{1}{2}$.   Suppose $ (\beta, \alpha) \in \mathcal{C}$.  If $ \beta > \frac{1}{2}$ then testing (\ref{alphabeta}) on $ u:=E^\beta$ shows that $ \alpha \le \beta - \beta^2$.   If $ \beta \le \frac{1}{2}$ then testing (\ref{alphabeta}) on $ u:=E^t $ and sending $ t \searrow \frac{1}{2}$  shows that $ \alpha \le \frac{1}{4}$.     Now for the other inclusion.   \\
  
 Fix $ t \ge 1 $ and put $ E_2:= E^t$.   Then we have 
  \begin{equation*}  \frac{ | \nabla E_2 |_A^2}{E_2^2}= t^2 \frac{ | \nabla E|_A^2}{E^2}, \qquad \frac{\L_A(E_2)}{E_2} = t(1-t) \frac{ | \nabla E|_A^2}{E^2} + t \frac{ \L_A(E)}{E}.
  \end{equation*}  Putting $ E=E_2 $ into (\ref{uu}) we obtain
  \begin{equation} \label{zop} \int_\Omega | \nabla u|_A^2dx \ge ( \frac{t}{2} - \frac{t^2}{4}) \int_\Omega \frac{ | \nabla E|_A^2}{E^2} u^2 dx + \frac{t}{2} \int_\Omega \frac{u^2}{E} d \mu, 
  \end{equation}
    and so we see that $ ( \frac{t}{2},  \frac{t}{2} - \frac{t^2}{4} ) \in \mathcal{C} $ for all $ t \ge 1$.   From this we see that the curve $ \alpha = \beta - \beta^2 $ for $ \beta \ge \frac{1}{2}$ is contained in $ \mathcal{C}$.    It is straightforward to see the remaining portion of $ \partial \mathcal{C}'$ is contained in $ \mathcal{C}$.  \\ To see the inequality does not attain when $ (\beta, \alpha) \in \partial \mathcal{C} \backslash \Gamma$ use the fact that (\ref{u}) does not attain in $H_0^1$ and the fact that $ \mu \ge 0$.    To see the inequality does attain on the remaining portion of $ \partial \mathcal{C}$ note that (\ref{zop}) attains at $ u:=E^\frac{t}{2} \in H_0^1(\Omega)$ for $ t > 1$. 
 \end{proof}  
 
 We now give a result relating to the first eigenvalue of $\L_A$ on subdomains of $ \Omega$.  
 Suppose $(E,\lambda_A(\Omega))$ is the first eigenpair (with $E>0$) of $ \L_A $ on $H_0^1(\Omega)$ and for $ B \subset \Omega$ we let $ \lambda_A(B)$ denote the first eigenvalue of $ \L_A $ on $H_0^1(B)$.

 \begin{cor} Let $E$ be as above.  For $B \subset \Omega$ we set 
 \[ \underline{\alpha}(B):= \inf_B \frac{| \nabla E|_A^2}{E^2}, \qquad  \overline{\alpha}(B):= \sup_B \frac{| \nabla E|_A^2}{E^2}.\]
 \flushleft 
 (i) \quad If $ \underline{\alpha}(B) > \lambda_A(\Omega)$ then 
 \[ 4 \lambda_A(B) \ge \frac{ (\underline{\alpha}(B)+\lambda_A(\Omega))^2}{\underline{\alpha}(B)}.\] 
 (ii) \quad If $ \lambda_A(\Omega) > \overline{\alpha}(B)$ then 
 \[ 4 \lambda_A(B) \ge \frac{ (\overline{\alpha}(B)+\lambda_A(\Omega))^2}{\overline{\alpha}(B)}. \]

 \end{cor}
 
 \begin{proof} Let $ B \subset \Omega$ and let $ u \in C_c^\infty(B)$ with $ \int_Bu^2=1$.  Using (\ref{zop}) gives 
 \[ 2 \int_B | \nabla u|_A^2dx \ge ( t - \frac{t^2}{2}) \inf_B \frac{| \nabla E|_A^2}{E^2} + \lambda_A(\Omega)t, \] for $ 0 < t < 2$. If $ t >2 $ then we get the same expression but with the infimum replaced with supremum.   Now take the infimum over $u$ and in case (i) set $ t:= 1 + \frac{ \lambda_A(\Omega)}{ \underline{\alpha}(B)} <2$ and in case (ii)  set $ t:=  1 + \frac{ \lambda_A(\Omega)}{ \overline{\alpha}(B)} >2 $ to see the result.
 
 \end{proof}

 \subsection{Weighted versions}
  We now examine weighted versions of the above inequalities which, as mentioned earlier, can be seen as analogs of Cafferelli-Kohn-Nirenberg inequalities.     We now introduce the spaces we work in.   
  
  \begin{dfn} For $ t \in \IR$ we define $ \| u \|_t^2:= \int_\Omega E^{2t} | \nabla u|_A^2 dx$.    \\
  Suppose $ E$ is an interior weight on $ \Omega$.  We define $ X_t$ to be the completion of $ C_c^{0,1}(\Omega \backslash K)$ with respect to $ \| \cdot \|_t$.    In the case that $E$ is a boundary weight on $ \Omega$ we define $ X_t$ to be the completion of $ C_c^{0,1}(\Omega)$ with respect to the same norm.

  \end{dfn}  
  
  \begin{remark}   One should note that if $E$ is an interior weight on $ \Omega$ and $ t> \frac{1}{2}$ then $ X_t$ does not contain $C_c^\infty(\Omega)$.  To see this use (\ref{weight1}) to see that if $C_c^\infty(\Omega) \subset X_t$ then $ E^t \in H^1_{loc}(\Omega)$ which we know to be false.    For $ t < \frac{1}{2}$ we do have $ C_c^\infty(\Omega) \subset X_t$. 
  
  \end{remark}

  \begin{thm} \label{weight1th} Suppose  $ t \neq \frac{1}{2}$ and $ E $ an interior weight on $ \Omega$.  Then 
  \begin{equation} \label{weight1}
  \int_\Omega E^{2t} | \nabla u|_A^2dx \ge ( t - \frac{1}{2} )^2 \int_\Omega | \nabla E|_A^2 E^{2t-2} u^2dx, 
  \end{equation} for all $ u \in X_t$.  Moreover  the constant is optimal and not attained.
  
  \end{thm}

\begin{proof} Let $ t \neq 0, \frac{1}{2}$, $ u \in C_c^{0,1}(\Omega \backslash K)$ and define $ w:=E^t u  \in C_c^{0,1}(\Omega \backslash K)$.   Put $ w$ into 
\[ \int_\Omega | \nabla w|_A^2dx \ge \frac{1}{4} \int_\Omega \frac{| \nabla E|_A^2}{E^2} w^2dx, \] and re-group to obtain (\ref{weight1}).    We now show the constant is optimal. 
Let $ v_m \in C_c^{0,1}(\Omega \backslash K)$ be such that 
 \[ D_m:= \frac{\int_\Omega | \nabla v_m |_A^2dx }{ \int_\Omega \frac{| \nabla E|_A^2}{E^2}v_m^2dx} \rightarrow \frac{1}{4}.\]  Define $ u_m:=E^{-t} v_m \in X_t$.  A computation shows that 
 \[ \frac{ \int_\Omega E^{2t} | \nabla u_m|_A^2dx }{  \int_\Omega | \nabla E|_A^2 E^{2t-2} u_m^2dx}= D_m +t^2-t, \] and since $ D_m \rightarrow \frac{1}{4}$ we see that $ (t - \frac{1}{2})^2$ is optimal.  \\   For the case $ \gamma:=\min_{\pOm}E>0$ we can show the constant is not obtained by using later results on improvements.  If $ \gamma=0$ we then sub $w$ into (\ref{start}) instead of (\ref{hardy}) and hold onto the extra term 
 \[ \int_\Omega E | \nabla ( E^{t- \frac{1}{2}} u )|_A^2dx\]  to see the optimal constant is not attained. 

\end{proof}

   \begin{thm} \label{boundaryweight} 
   \flushleft
   (i) \quad  Suppose $ 0 \neq t < \frac{1}{2}$ and $ E $ is a boundary weight on $ \Omega$.    Then 
   \begin{equation} \label{noim}
   \int_\Omega E^{2t} | \nabla u|_A^2dx - ( t - \frac{1}{2})^2 \int_\Omega | \nabla E|_A^2 E^{2t-2} u^2dx \ge 0,
   \end{equation} for all $ u \in X_t$.  Moreover the constant is optimal and not attained.  \\
   (ii) \quad   Suppose $ 0 \neq t < \frac{1}{2}$ and $ E $ is a boundary weight on $ \Omega$.     Then 
   \begin{equation} \label{im}
   \int_\Omega E^{2t} | \nabla u|_A^2dx - ( t - \frac{1}{2})^2 \int_\Omega | \nabla E|_A^2 E^{2t-2} u^2dx \ge ( \frac{1}{2}-t) \int_\Omega E^{2t-1} u^2 d \mu, 
   \end{equation} for all $ u \in X_t$.  Moreover the constant on the right is optimal and not attained.   \\
   (iii) \quad  Suppose $ t > \frac{1}{2}$ and $ E \in  L^\infty(\Omega)$ is a boundary weight on $ \Omega$.    Then 
   \[ \inf \left\{  \frac{\int_\Omega E^{2t} | \nabla u|_A^2dx}{ \int_\Omega | \nabla E|_A^2 E^{2t-2} u^2dx}: \; \;  u \in X_t \backslash \{0 \}  \right\} =0. \]  
    
   \end{thm}

 \begin{proof}  We first prove (\ref{im}) for $ u \in C_c^{0,1}(\Omega)$ which then gives us (\ref{noim}) for the same  class of $u$'s.    Suppose $ 0 \neq t < \frac{1}{2}$ and $ E $ is a boundary weight on $ \Omega$.   We now use the notation introduced in the proof of lemma \ref{startingpoint}; namely $ E_\E $ is the standard mollification of $ E $ and $ F_\E:=\L_A(E_\E)$.  Recall that for any $ u \in C_c^{0,1}(\Omega)$ we have  $ uF_\E \rightarrow u \mu $ in $ H^{-1}(\Omega)$ and that we have 
  \[
  \int_\Omega | \nabla v|_A^2dx \ge \frac{1}{4} \int_\Omega \frac{| \nabla E_\E|_A^2}{E_\E^2} v^2dx + \frac{1}{2} \int_\Omega \frac{v^2}{E_\E}F_\E dx, \] for all $ v \in H_0^1(\Omega)$.   Now let $ u \in C_c^{0,1}(\Omega) $ and set $ v:=E_\E^t u \in C_c^{0,1}(\Omega)$.  Putting $ v $ into the above gives 
  \begin{equation} \label{hold}
  \int_\Omega E_\E^{2t} | \nabla u|_A^2dx\ge ( t - \frac{1}{2})^2 \int_\Omega | \nabla E_\E|_A^2 E_\E^{2t-2} u^2dx + (\frac{1}{2}-t) \int_\Omega E_\E^{2t-1} u^2 F_\E dx.  
  \end{equation}  Now since $ E_\E^{2t} \rightarrow E^{2t} $ in $L^1_{loc}(\Omega)$ we have 
  \[ \int_\Omega E_\E^{2t} | \nabla u|_A^2 dx \rightarrow \int_\Omega E^{2t} | \nabla u|_A^2 dx, \] and using similar ideas from the proof of lemma \ref{startingpoint} one can show that 
  \[  \int_\Omega E_\E^{2t-1} u^2 F_\E dx \rightarrow \int_\Omega E^{2t-1} u^2 d \mu.\]   So using these results, sending $ \E \searrow 0 $ in (\ref{hold}) and after an application of Fatou's lemma we arrive at (\ref{im}) for $ u \in C_c^{0,1}(\Omega)$.    \\ Now we show the constants are optimal.  Recalling the proof of theorem \ref{first} there exists $ v_m \in C_c^\infty(\Omega)$ such that 
  \[ D_m:= \frac{\int_\Omega | \nabla v_m|_A^2dx }{\int_\Omega \frac{ | \nabla E|_A^2}{E^2}v_m^2dx} \rightarrow \frac{1}{4}, \qquad F_m:= \frac{\int_\Omega | \nabla v_m|_A^2dx - \frac{1}{4} \int_\Omega \frac{ | \nabla E|_A^2}{E^2}v_m^2dx}{ \int_\Omega \frac{v_m^2}{E} d \mu} \rightarrow \frac{1}{2}.\]  Define $ u_m:=E^{-t} v_m $ which one easily sees is an element of $ X_t$.   Then 
  \[ \Phi_m:=\frac{\int_\Omega E^{2t} | \nabla u_m|_A^2dx}{ \int_\Omega | \nabla E|_A^2 E^{2t-2} u_m^2dx} = D_m +t^2 -2t   \frac{\int_\Omega E^{-1} v_m \nabla v_m \cdot A \nabla E dx }{ \int_\Omega \frac{ | \nabla E|_A^2}{E^2}v_m^2dx}, \qquad \mbox{and} \]  
  \[ \Psi_m:=\frac{\int_\Omega E^{2t} | \nabla u_m|_A^2dx - (t- \frac{1}{2})^2 \int_\Omega | \nabla E|_A^2 E^{2t-2} u_m^2dx}{ \int_\Omega E^{2t-1} u_m^2 d \mu} = F_m + \frac{ t \int_\Omega \frac{| \nabla E|_A^2}{E^2} v_m^2dx -2t \int_\Omega E^{-1} v_m \nabla v_m \cdot A \nabla Edx}{ \int_\Omega \frac{v_m^2}{E} d \mu}.\]  
    Using $ E_\E, F_\E$ as defined above one can show, using similar methods, that 
  \begin{equation} \label{tttt} 2 \int_\Omega E^{-1} v_m \nabla v_m \cdot A \nabla Edx= \int_\Omega \frac{v_m^2}{E} d \mu + \int_\Omega \frac{ | \nabla E|_A^2}{E^2} v_m^2dx. 
  \end{equation}  So from this we see that  
  \[ \Phi_m= D_m+t^2-t - t \frac{ \int_\Omega \frac{v_m^2}{E} d \mu }{ \int_\Omega \frac{ | \nabla E|_A^2}{E^2} v_m^2dx},  \]  and noting that 
  \[ \frac{ \int_\Omega \frac{v_m^2}{E} d \mu }{ \int_\Omega \frac{ | \nabla E|_A^2}{E^2} v_m^2dx} = \frac{D_m - \frac{1}{4}}{F_m} \rightarrow 0 ,\] we see that (\ref{noim}) is optimal.   Similarly one sees using (\ref{tttt}) that $ \Psi_m = F_m-t $ and hence (\ref{im}) is optimal.  \\
To show the constants are not obtained we as usual hold on to the extra term  that we dropped in the above calculations.  Since $ \int_\Omega E^{-1} | \nabla E|_A^2dx = \infty$ one can show this extra term is positive for $ u \in X_t \backslash \{0\}$.  \\
  \noindent (iii) \quad  Now take $ t > \frac{1}{2}$ and $ E $ a boundary weight on $ \Omega$.   For $ \E, \tau > 0$ but small define 
  \begin{equation*}
  u_{\E, \tau}(x):= \left\{ \begin{array}{ll}
  0 & \qquad E < \E \\
  E^\tau - \E^\tau & \qquad  E > \E.
  \end{array}
  \right.
  \end{equation*}   Then $ u_{\E,\tau} \in X_t$.  Now use the sequence $ u_m$ where $ u_m:=u_{\E_m, \tau_m}$ to see desired result where $ \E_m:=m^{-m}$ and $ \tau_m:=m^{-1}$.
 
 \end{proof}

  \subsection{More general weighted inequalities}
  
  We now investigate the possibility of inequalities of the form 
  \[ \int_\Omega W(x) | \nabla u|_A^2dx \ge \int_\Omega U(x) u^2dx, \qquad u \in C_c^{0,1}(\Omega \backslash K).\]   
 \begin{thm} Suppose $ E $ is an interior weight on  $\Omega$ with $ \gamma:=min_{\pOm} E$ and $ 0 < f \in C^\infty(\gamma,\infty)$.  Then 
 \begin{equation} \label{gweight}
 \int_\Omega f(E)^2 | \nabla u|_A^2dx \ge \int_\Omega | \nabla E|_A^2 \left( \frac{f(E)^2}{4E^2} +f(E) f''(E) \right) u^2dx,
 \end{equation}
 for all $ u \in C_c^{0,1}(\Omega \backslash K)$.    In addition this is optimal (in the sense that the optimal constant is $1$) if $ \liminf_{z \rightarrow \infty} f''(z) > 0$  or if  $ \lim_{z \rightarrow \infty} \frac{z^2 f''(z)}{f(z)}=0$.

 \end{thm}

  \begin{proof} Let $ u \in C_c^{0,1}(\Omega \backslash K)$ and define $ w:= f(E) u \in C_c^{0,1}(\Omega \backslash K)$.  Putting $ w $ into (\ref{hardy}), integrating by parts and re-grouping gives (\ref{gweight}).   Let $ v_m \in C_c^{0,1}(\Omega \backslash K)$ be such that 
  \[ D_m:= \frac{ \int_\Omega | \nabla v_m|^2dx }{  \int_\Omega \frac{| \nabla E|_A^2}{E^2}v_m^2dx} \rightarrow \frac{1}{4}.\] Without loss of generality we can assume the supports of $ v_m$ concentrate on $K$.  Define $ u_m:= \frac{v_m}{f(E)} \in C_c^{0,1}(\Omega \backslash K)$.  Then a computation shows that 
  \begin{eqnarray*}
  Q_m&:=& \frac{\int_\Omega f(E)^2 | \nabla u_m|_A^2dx }{ \int_\Omega | \nabla E|_A^2 \left( \frac{f(E)^2}{4E^2} +f(E) f''(E) \right) u_m^2dx } \\
  &=& \frac{\int_\Omega | \nabla v_m|_A^2dx + \int_\Omega \frac{| \nabla E|_A^2 f''(E)}{f(E)^2} v_m^2dx}{   \int_\Omega \frac{| \nabla E|_A^2}{4E^2}v_m^2dx + \int_\Omega \frac{| \nabla E|_A^2 f''(E)}{f(E)^2} v_m^2dx }.
  \end{eqnarray*}
  Now suppose $ \liminf_{z \rightarrow \infty} f''(z) > 0$.  Then using the monotonicity of $ x \mapsto \frac{\alpha +x}{\beta+x}$, where $ \alpha $ and $ \beta$ are positive constants, shows $ Q_m \rightarrow 1$.   Now suppose $ \lim_{z \rightarrow \infty} \frac{z^2 f''(z)}{f(z)}=0$.   Using this and the fact that the $ v_m$'s support concentrates on $K$ one easily sees that 
  \[ \frac{\int_\Omega \frac{| \nabla E|_A^2 f''(E)}{f(E)^2} v_m^2dx}{ \int_\Omega \frac{| \nabla E|_A^2}{4E^2}v_m^2dx } \rightarrow 0.\] Using this one sees that $ Q_m \rightarrow 1$.

  \end{proof}

   \subsection{Improvements}
   We now investigate the possibility of improving (\ref{hardy}) in the sense of potentials.  The method we employ was first used by Ghoussoub and Moradifam (see [GM]).    We now define precisely what we mean by a potential.  Suppose $ E $ is an interior weight on $ \Omega$ and $ 0 \le V \in C^\infty(\Omega \backslash K)$ (recall $K$ is the support of $ \mu$).  We say $V$ is a potential for $E$ provided 
   \begin{equation} \label{potential}
    \int_\Omega | \nabla u|_A^2dx - \frac{1}{4} \int_\Omega \frac{| \nabla E|_A^2}{E^2} u^2dx \ge \int_\Omega V(x) u^2dx, 
    \end{equation} for all $ u \in H_0^1(\Omega)$.    We analogously define a potential $V$ for the case that $ E $ is a boundary weight on $ \Omega$ except we restrict our attention to $ 0 \le V \in C^\infty(\Omega)$.   The next theorem gives necessary and sufficient conditions for $V$ to be a potential of $E$ in terms of solvability of a  singular linear equation.    For the necessary direction we will need to assume some conditions on $ \Omega$.   \\
    \textbf{(B1)} \quad  
    Suppose $ E $ is an interior weight on $ \Omega$.  We  assume that that there exists a sequence $(\Omega_m)_m$ of non-empty subdomains of $ \Omega$  which are connected, have a smooth boundary, $  \Omega_m \subset \subset \Omega \backslash K$, $ \Omega_m \subset \subset \Omega_{m+1}$ and $ \Omega \backslash K = \cup_m \Omega_m$.  \\
    \textbf{(B2)} \quad   Suppose $ E $ is a boundary weight on $ \Omega$.  We assume that there exists a sequence $ (\Omega_m)_m$ of non-empty subdomains of $ \Omega$ which are connected, have a smooth boundary, $ \Omega_m \subset \subset \Omega_{m+1}$ and $ \Omega = \cup_m \Omega_m$.

   \begin{thm} \label{geninteriorimprov} \textbf{(interior improvements)} \label{interior improvements} Suppose $ E $ is an interior weight on $ \Omega$ and $ 0 \le V \in C^\infty(\Omega \backslash K)$. \\
   \flushleft (i) \quad Suppose there exists some $ 0 < \phi \in C^2(\Omega \backslash K)$ such that 
   \begin{equation} \label{equation_1}
   -\L_A(\phi) + \frac{A \nabla E \cdot \nabla \phi}{E} +V \phi \le 0 \qquad \mbox{ in $ \Omega \backslash K$}. 
   \end{equation} Then $ V$ is a potential for $E$.  After the change of variables $ \theta:=E^\frac{1}{2} \phi$ one sees that it is sufficient to find a $ 0 < \theta \in C^2(\Omega \backslash K)$ such that 
   \begin{equation} \label{thetaequation}
   \frac{ -\L_A(\theta)}{\theta} + \frac{| \nabla E|_A^2}{4 E^2} + V \le 0 \qquad \mbox{ in $ \Omega \backslash K$}.
   \end{equation} 
   (ii) \quad Suppose $ V $ is a potential for $E$ and $ \Omega$ satisfies (B1).  Then there exists some $ 0 < \theta \in C^\infty(\Om \backslash K)$ which satisfies (\ref{thetaequation}).

   \end{thm}

 It is important to note that the above theorem can be used (in theory) for best constant calculations; without the need for constructing appropriate minimizing sequences.  To see this suppose $ 0 \le V$ is a potential for the interior weight $E$ and let $C(V)>0$ denote the associated best constant, ie
 \[ C(V):= \inf \left\{  \frac{\int_\Omega | \nabla u|_A^2dx- \frac{1}{4} \int_\Omega \frac{| \nabla E|_A^2}{E^2}u^2dx}{ \int_\Omega Vu^2dx} \; : \; u \in H_0^1(\Omega) \backslash \{0\}  \right\}. \]    Then one sees that 
 \[C(V)=\sup \left\{ c>0: \exists 0< \theta \in C^2(\Omega \backslash K) \; \; s.t. \; \; \frac{ -\L_A(\theta)}{\theta} + \frac{| \nabla E|_A^2}{4 E^2} + cV \le 0 \qquad \mbox{ in $ \Omega \backslash K$} \right\}. \]     After theorem \ref{odemethod}, which is analogous result to the above theorem but phrased in terms of solvability of a linear ode, this remark on best constants will be of more importance because of the shear magnitude of results concerning solvability of ode's.

   \begin{thm}  \label{boundaryimprovements} (\textbf{boundary improvements}) Suppose $ E $ is a boundary weight on  $\Omega$ and $ 0 \le V \in C^\infty(\Omega)$.   
   \flushleft
   (i) \quad  Suppose $ E \in C^{0,1}(\Ov)$, $V$ is a potential for $E$ and  $ \Omega $ satisfies (B2).  Then there exists some $ 0 < \theta \in C^{1,\alpha}(\Omega) $ for all $ \alpha < 1 $ such that 
   \begin{equation} \label{boundeq1} 
   \frac{-\L_A(\theta)}{\theta} + \frac{ | \nabla E|_A^2}{4 E^2} + V \le 0 \qquad \mbox{ in $ \Omega$}.
   \end{equation}
   
   (ii) \quad   Suppose there exists some $ 0 < \phi \in C^2(\Omega)$ such that 
   \begin{equation} \label{bound-eq}
   \frac{-\L_A(\phi)}{\phi} + \frac{ A \nabla E \cdot \nabla \phi}{E \phi} - \frac{ \mu}{2E} + V \le 0 \qquad \mbox{ in $ \Omega$.}
   \end{equation}  Then $V$ is a potential for $E$.
    \end{thm} 
   
  \begin{remark}
    Note that putting $ \theta := E^\frac{1}{2} \phi$ into (\ref{bound-eq}) gives, at least formally, (\ref{boundeq1}).  Also one can replace $ \mu$ by the absolutely continuous part of  $ \mu$ in (\ref{bound-eq}).
   
 \end{remark}

  \noindent \textbf{Proof of theorem \ref{geninteriorimprov}.}    (i) \quad Suppose $ V \in C^\infty(\Omega \backslash K) $ is non-negative and there exists some  $ 0 < \phi \in C^2(\Omega \backslash K)$ which solves (\ref{equation_1}).  Let $ u \in C_c^{0,1}(\Omega \backslash K)$ and define $ v:=E^\frac{-1}{2} u$ so   by lemma \ref{startingpoint} we have 
\[ \int_\Omega | \nabla u|_A^2dx - \frac{1}{4} \int_\Omega \frac{ | \nabla E|_A^2}{E^2} u^2dx = \int_\Omega E | \nabla v|_A^2dx. \]   Now define $ \psi \in C_c^{0,1}(\Omega \backslash K)$ by $ v:= \phi \psi$.     A calculation shows that 
\begin{equation} \label{pp}
E | \nabla v|_A^2 = E \psi^2 | \nabla \phi |_A^2 + E \phi^2 | \nabla \psi|_A^2 + 2 E \phi \psi A \nabla \phi \cdot \nabla \psi, 
\end{equation} and integrating, by parts, the last term over $ \Omega $  we obtain 
\begin{eqnarray*}
 \int_\Omega \psi^2 E | \nabla \phi |_A^2dx + 2 \int_\Omega \phi \psi E A \nabla \phi \cdot \nabla \psi dx &=& \int_\Omega \psi^2 \left( \L_A(\phi) \phi E - \phi A \nabla E \cdot \nabla \phi \right) dx \\
&=& \int_\Omega u^2 \left( \frac{ \L_A(\phi)  - \frac{ A \nabla E \cdot \nabla \phi}{E} }{\phi} \right) dx\\
&=:& Q,
\end{eqnarray*} but by (\ref{equation_1}) $ Q \ge \int_\Omega V(x) u^2 dx $ and so we see 
\[ \int_\Omega | \nabla u|_A^2dx - \frac{1}{4} \int_\Omega \frac{ | \nabla E|_A^2}{E^2} u^2 dx \ge \int_\Omega E \phi^2 | \nabla \psi |_A^2dx + \int_\Omega V u^2 dx, \] for all $ u \in C_c^{0,1}(\Omega \backslash K)$.   Now since $ C_c^{0,1}(\Omega \backslash K)$ is dense in $H_0^1(\Omega)$  and using  Fatou's lemma one can  show (\ref{potential}) holds for all $ u \in H_0^1(\Omega)$.      \\ 

\noindent (ii) \quad Now suppose $ V \in C^\infty(\Omega \backslash K)$ is a  potential for $E$ and $ (\Omega_m)_m$  is the sequence of domains from assumption (B1). Define the elliptic operator $P$ by \[ P(u):= \L_A(u) - \frac{| \nabla E|_A^2}{4 E^2} u -Vu. \]  Using a standard constrained minimization argument along with the strong maximum principle there exists some $ 0 < \theta_m \in H_0^1(\Omega_m)$ such that 
\begin{eqnarray} \label{thingy}
P(\theta_m) &=& \lambda_m \theta_m \qquad \mbox{ in $ \Omega_m$} \nonumber \\
\theta_m &=& 0 \qquad \quad  \mbox{ on $ \pOm_m$,}
\end{eqnarray} where $ 0 \le \lambda_m$, ie. $ (\theta_m, \lambda_m)$ is the first eigenpair of $ P$ in $H_0^1(\Omega_m)$.   Since $H_0^1(\Omega_m) \subset H_0^1(\Omega_{m+1})$ we see that $ \lambda_m$ is decreasing and hence there exists some $ 0 \le \lambda $ such that $ \lambda_m \searrow \lambda$.   Let $ x_0 \in \cap_m\Omega_m$ and suitably scale $ \theta_m$ such that $ \theta_m(x_0)=1$ for all $m$. 
   Now fix $k$ and let $  m > k+1 $.   Then \[ P(\theta_m) - \lambda_m \theta_m = 0 \qquad \mbox{ in $ \Omega_{k+1}$}, \] and we now apply Harnacks inequality to the operator $P-\lambda_m$ to see there exists some $ C_k $ such that 
\[ \sup_{\Omega_k} (\theta_m) \le C_k \inf_{\Omega_k} (\theta_m) \le C_k.\]    So we see that $ ( \theta_m ) $ is bounded in $ L^\infty_{loc}( \Omega \backslash K)$.  Now applying elliptic regularity theory and a bootstrap argument one sees that $ (\theta_m)_{m > k+1} $ 
is bounded in $ C^{1,\alpha}(\Omega_k)$ for  $ \alpha <1 $ and after applying a diagonal argument one sees that there exists some  non-zero $0 \le \theta \in C^{1,\alpha}(\Omega \backslash K)$  such that $ \theta_m \rightarrow \theta $ in $ C^{1,\alpha}(\Omega_k)$ for all $k$.  Using this convergence one can pass to the limit in (\ref{thingy}) to see that $ P(\theta)=\lambda \theta $ in $ \Omega \backslash K$ and after applying the strong maximum principle on $ \Omega_m$ one sees that $ \theta >0 $ in $ \Omega \backslash K$.   Now applying regularity theory one sees that $ \theta \in C^\infty(\Omega \backslash K)$.

 \hfill  $\Box$

 \noindent \textbf{Proof of theorem \ref{boundaryimprovements}. }   (i) \quad  The proof is essentially unchanged from the proof of theorem \ref{geninteriorimprov}. \\
  \noindent (ii) \quad Again the proof is the same as in theorem \ref{geninteriorimprov} except now the measure $ \mu$ does not drop out. 
\hfill $\Box $   \\

The next theorem gives some explicit examples of potentials.

  \begin{thm} \label{exampleofimprovements}   
  (i) \quad Suppose $ E $ is an interior weight on $ \Omega$,  $ 0 < \gamma:=\min_{\pOm} E$ and $ 0 < f \in C^2( (\gamma, \infty))$.  Then for all $ u \in C_c^{0,1}(\Omega \backslash K)$ we have 
  \[ \int_\Omega | \nabla u|_A^2 dx - \frac{1}{4} \int_\Omega \frac{| \nabla E|_A^2}{E^2}u^2 dx \ge \int_\Omega \frac{| \nabla E|_A^2}{f(E)} \left( - f''(E) - \frac{f'(E)}{E} \right) u^2 dx . \]  In particular by taking $ f(E):= \sqrt{ \log( \gamma^{-1}E)}$ we obtain    
  \begin{equation} \label{interiorexamp}
  \int_\Omega | \nabla u|_A^2dx - \frac{1}{4} \int_\Omega \frac{| \nabla E|_A^2}{E^2}u^2 dx \ge \frac{1}{4} \int_\Omega \frac{ | \nabla E|_A^2}{E^2 \log^2(\gamma^{-1} E)} u^2 dx,
  \end{equation} for all $ u \in H_0^1(\Omega)$.    Now suppose  $ 0 < \gamma =E $ on $ \pOm$.  Then $ \frac{1}{4}$ (on the right hand side of (\ref{interiorexamp}))  is optimal.  \\
  (ii) \quad Suppose $ E \in  L^\infty(\Omega)$ is a boundary weight.  Then 
  \begin{equation} \label{bexam}
  \int_\Omega | \nabla u|_A^2dx  - \frac{1}{4} \int_\Omega \frac{| \nabla E|_A^2}{E^2}u^2dx \ge \frac{1}{4} \int_\Omega \frac{ | \nabla E|_A^2}{  E^2 \log^2\left( \frac{E}{e \| E \|_{L^\infty}} \right)} u^2 dx,
  \end{equation} for all $ u \in H_0^1(\Omega)$.  \\

  \end{thm}

   \begin{proof}
(i) \quad  Let $ E $ be an interior weight on $ \Omega$, $ \gamma:=\min_{\pOm}E >0 $ and suppose $ 0 < f \in C^2((\gamma, \infty))$.  Put $ \phi:=f(E)$ into (\ref{equation_1}) to obtain the result.      \\
Now take $ f(E):=\sqrt{ \log(\gamma^{-1}E)}$ to obtain (\ref{interiorexamp}) for all $ u \in C_c^{0,1}(\Omega \backslash K)$ and extend to all of $H_0^1(\Omega)$ by density and by Fatou's lemma.    We now show $ \frac{1}{4}$ is optimal. \\   Fix $ 0 < t < \frac{1}{2}$ and for $ \tau > \frac{1}{2}$ define $ u_\tau:= E^t \log^\tau(\gamma^{-1}E)$.   By lemma \ref{test} $ u_\tau \in H_0^1(\Omega)$. A  computation shows that 
\begin{eqnarray*} 
 \frac{ \int_\Omega | \nabla u_\tau|_A^2 dx - \frac{1}{4} \int_\Omega \frac{ | \nabla E|_A^2}{E^2} u_\tau^2 dx}{ \int_\Omega \frac{ | \nabla E|_A^2}{ E^2 \log^2( E \gamma^{-1})} u_\tau^2 dx }& = & ( t^2- \frac{1}{4}) \frac{ \int_\Omega E^{2t-2} | \nabla E|_A^2 \log^{2 \tau}(E \gamma^{-1}) dx  }{ \int_\Omega E^{2t-2} | \nabla E|_A^2 \log^{2 \tau-2}(E \gamma^{-1}) dx } \\   
 && + \tau^2  \\ 
 &&  2 t \tau \frac{ \int_\Omega E^{2t-2} | \nabla E|_A^2 \log^{2 \tau-1}(E \gamma^{-1}) dx }{\int_\Omega E^{2t-2} | \nabla E|_A^2 \log^{2 \tau-2}(E \gamma^{-1}) dx} \\
 &=& (t^2-\frac{1}{4} ) \frac{ J_t(\tau+1)}{J_t(\tau)} + \tau^2 + 2 t \tau \frac{ J_t(\tau+1/2)}{J_t(\tau)},
 \end{eqnarray*} where $ J_t(\tau)$ is defined in lemma \ref{test}.  Sending $ \tau \searrow \frac{1}{2}$ and using results from lemma \ref{test} we see $ \frac{1}{4}$ is optimal.   \\
 \noindent (ii) \quad  Suppose $ E \in L^\infty(\Omega)$ is a boundary weight on $ \Omega$. Here we use the notation from the proof of lemma \ref{startingpoint}; $ E_\E:= \eta_\E \ast E$,  $ F_\E:= \L_A(E_\E)$.     Let $ 0 < f \in C^2( (0, \| E \|_{L^\infty} ])$.   Then starting at (\ref{spbound})  for $ E_\E $ and decomposing $ v$ as usual one arrives at 
 \begin{eqnarray}
  \int_\Omega | \nabla  u |_A^2dx - \frac{1}{4} \int_\Omega \frac{ | \nabla E_\E|_A^2}{E_\E^2} u^2dx & \ge & \int_\Omega \frac{ | \nabla E_\E|_A^2}{f(E_\E)} \left( - f''(E_\E) - \frac{f'(E_\E)}{E_\E} \right) u^2 dx\\
  && + \int_\Omega \left( \frac{ f'(E_\E)}{f(E_\E)} + \frac{1}{2 E_\E} \right) u^2 F_\E dx, 
  \end{eqnarray} for all $ u \in C_c^\infty(\Omega)$ after using methods similar to the proof of (i).    Now take $ f(z):=\sqrt{  -\log( \frac{z}{e \| E \|_{L^\infty}})}$ and let $ u \in C_c^\infty(\Omega)$.  Then one has 
  \[ \int_\Omega |  \nabla u|_A^2dx - \frac{1}{4} \int_\Omega \frac{ | \nabla E_\E|_A^2}{E_\E^2} u^2dx \ge \frac{1}{4} \int_\Omega \frac{| \nabla E|_A^2}{ E_\E^2 \log^2(  \frac{ E_\E}{e \| E \|_{L^\infty}})} u^2 dx + I_\E, \] where 
  \[ I_\E:= \frac{1}{2} \int_\Omega \frac{u^2}{E_\E^2} \left( 1 + \frac{1}{ \log(  \frac{ E_\E}{e \| E \|_{L^\infty}})} \right) F_\E dx.\] 
Using methods similar to ones used in the proof of lemma \ref{startingpoint} one easily sees that $ \lim_{\E \searrow 0} I_\E \ge 0 $.   Using this and standard  results on convolutions and Fatou's lemma we obtain the desired inequality for $ u \in C_c^\infty(\Omega)$ and we then extend to all of $ H_0^1(\Omega)$.  

\end{proof}

  We now obtain a more useful (than (\ref{thetaequation})) necessary and sufficient condition for $V$ to be a potential for $E$; at least in the case where $ E $ is an interior weight on $ \Omega$ and $ E=\gamma \ge 0$ on $ \pOm$.     As in theorem \ref{geninteriorimprov} we assume some geometrical properties of $ \Omega$.

  \begin{thm} (\textbf{Interior improvements using ode methods}) \label{odemethod}  Suppose $ E $ is an interior weight on $ \Omega$, $ E=\gamma \ge 0 $ on $ \pOm$, $ 0 \le  f \in C^\infty(\gamma, \infty)$ and $ \Omega_t:=\{x \in \Omega: \gamma + \frac{1}{t} <E(x) < t\}$ is connected for sufficiently large $t$.   Then the following are equivalent: \\
  (i) \quad For all $ u \in H_0^1(\Omega)$ 
  \begin{equation} \label{pep}
  \int_\Omega | \nabla u|_A^2dx - \frac{1}{4} \int_\Omega \frac{| \nabla E|_A^2}{E^2} u^2dx \ge \int_\Omega f(E) | \nabla E|_A^2 u^2dx. 
  \end{equation}  
  (ii) \quad There exists some $ 0 < h \in C^2(\gamma, \infty)$ such that 
  \begin{equation}
  h''(t)+ \left( f(t)+ \frac{1}{4t^2} \right) h(t) \le 0, 
  \end{equation} in $ ( \gamma, \infty)$.

  \end{thm}

  \begin{proof} Let $ E $ be an interior weight on $ \Omega$, $ E = \gamma \ge 0 $ on $ \pOm$ and $ 0 \le f \in C^\infty(\gamma ,\infty)$. \\
 \noindent $ (ii) \Rightarrow (i)$ \\  Setting $ \theta:=h(E)$ and using (ii) along with theorem \ref{geninteriorimprov} gives (i).\\  
\noindent $ (i) \Rightarrow (ii)$. \\ 
 The proof will be similar to theorem \ref{geninteriorimprov} (ii).  
  Let $ \gamma < t_m \nearrow \infty$ and  define $ \Omega_m:=\{x \in \Omega: \gamma + \frac{1}{t_m} < E(x) < t_m\}$.  By hypothesis we can take $ \Omega_m$ to be connected and non-empty for each $m$. Now define  $ H_{0,E}^1(\Omega_m):=\{ \phi \in H_0^1(\Omega_m): \phi \; \mbox{ is constant on level sets of $E$} \}$  and set  
 \[ F(\phi):=\frac{1}{2} \int_{\Omega_m} | \nabla \phi |_A^2dx, \quad  J(\phi):=\frac{1}{2} \int_{\Omega_m} | \nabla E|_A^2(f(E) + \frac{1}{4E^2}) \phi^2 dx, \quad  M_m:=\{ \phi \in H_{0,E}^1(\Omega_m): J(\phi)=2^{-1} \}.\]  Standard methods show the existence of $ 0 < \phi_m \in H_{0,E}^1(\Omega_m)$ such that $ \lambda_m:=\inf_{M_m} F= F(\phi_m)$  and hence $ \L_A(\phi_m)=\lambda_m | \nabla E|_A^2( f(E)+ \frac{1}{4E^2}) \phi_m$ in $ \Omega_m$ with $ \phi_m =0 $ on $ \partial \Omega_m$.  Since $ H_{0,E}^1(\Omega_m) \subset H_{0,E}^1(\Omega_{m+1})$  one sees that $ \lambda_m$ is
  decreasing and from (\ref{pep}) one sees that $ \lambda_m\ge 1$ and hence there exists some $ \lambda \ge 1 $ such that $ \lambda_m \searrow \lambda$.
  By suitably scaling $ \phi_m$ as before and after an application of Harnacks inequality we can assume that $ \phi_m \rightarrow \phi$ in $ C^{1,\alpha}_{loc}(\Omega \backslash K)$ where $ \phi \ge 0 $ is nonzero and constant on level sets of $E$.   Passing to the limit  shows that 
       \[ \L_A(\phi)= \lambda | \nabla E|_A^2 \left( f(E) + \frac{1}{4E^2} \right) \phi \qquad \mbox{in $ \Omega \backslash K$},\] and a strong maximum principle argument shows that $ \phi>0$ in $ \Omega \backslash K$.   Since $ \phi$ constant on level sets of $E$ we have $ \phi=h(E)$ for some $ 0<h$ in $ (\gamma, \infty)$ and   since $ \phi$ smooth on $ \Omega \backslash K$ we see that $ h$ is smooth on $(\gamma, \infty)$.  Writing the equation for $ \phi$ in terms of $h$ gives 
 \[ - h''(E) | \nabla E|_A^2 = \lambda h(E) \left( f(E) + \frac{1}{4E^2} \right) | \nabla E|_A^2 \qquad \mbox{ in $ \Omega \backslash K$}, \] and using Hopfs lemma we can cancel the gradients.

  \end{proof}

  Using the vast knowledge of ode's one can use the above theorem to obtain various results concerning potentials of the form $ V(x)=| \nabla E|_A^2 f(E)$.   We don't exploit this fact other than to look at one result. 
  
  \begin{cor} Suppose $ E $ is an interior potential on $ \Omega$ and $ E=0$ on $ \pOm$.  Then there no $0<f \in C(0,\infty)$   such that 
  \[ \int_\Omega | \nabla u|_A^2dx - \frac{1}{4} \int_\Omega \frac{| \nabla E|_A^2}{E^2} u^2dx \ge \int_\Omega f(E) | \nabla E|_A^2 u^2dx, \qquad u \in H_0^1(\Omega).\]

  \end{cor}

  \begin{proof} Suppose there is such a function $f$.  Using the proof of theorem \ref{odemethod} one sees that there is some $ 0 < h \in C^2(0,\infty)$ such that 
  \[ h''(t)+  \lambda \left( f(t)+ \frac{1}{4t^2} \right) h(t) =0,\]  in $ (0,\infty)$ where $ \lambda \ge 1$.  Now set $ h(t)=\sqrt{t} y(t)$ to see that 
  \[ 0=y''(t) + \frac{y'(t)}{t} + y(t) \left( \lambda f(t) + \frac{\lambda-1}{4 t^2} \right), \] in $ (0,\infty)$ and $ y(t)>0$.    But oscillation theory from ordinary differential equations shows this is impossible.

  \end{proof}

 Other than some regularity issues this ode approach extends immediately to the case where $ E $ is a boundary weight in $ \Omega$.        Using this corollary (but in the boundary case) one can show the result mentioned in the examples section  regarding improvements of  Hardy's boundary inequality in the half space; the regularity is not an issue in this example since $ \delta(x):=dist(x, \partial \IR^n_+)=x_n$ is smooth.     
 
   We now present a result obtained by  Avkhadiev and Wirths (see [AW]).    Given a domain $ \Omega$ in $ \IR^n$ we say it has finite inradius if $ \delta(x):=dist(x,\pOm)$ is bounded in $ \Omega$.   We let $ \lambda_0$ (Lambs constant) denote the first positive zero of $ J_0(t)-2tJ_1(t)$ where $ J_n$ is the Bessel function of order $n$.   Numerically one sees that $ \lambda_0=0.940...$.   Now for their result.

 \begin{thm}(Avkhadiev, Wirths)  Suppose $ \Omega$ is a convex domain in $ \IR^n$ with finite inradius.  Then 
 \[ \int_\Omega | \nabla u|^2dx \ge \frac{1}{4} \int_\Omega \frac{u^2}{\delta^2} dx+ \frac{\lambda_0^2}{\| \delta \|_{L^\infty}^2} \int_\Omega u^2dx, \qquad u \in H_0^1(\Omega)\] is optimal.

 \end{thm}

 This  extends a result of H. Brezis and M. Marcus (see [BM]) which said that if $\Omega$ is a convex subset of $ \IR^n$ then 
 \[ \int_\Omega | \nabla u|^2dx \ge \frac{1}{4} \int_\Omega \frac{u^2}{\delta^2}dx + \frac{1}{4 diam^2(\Omega)} \int_\Omega u^2dx, \qquad u \in H_0^1(\Omega)\]  where  $ diam(\Omega)$ denotes the diameter of $ \Omega$.   Note that there are unbounded convex domains with infinite diameter but finite inradius; for example take a cylinder.  
 
 We establish a generalized version of this result.    Suppose $ \mu$ is a nonnegative nonzero locally finite measure in $ \Omega$ (possibly unbounded)  and $ 0 <E \in L^\infty(\Omega)$ is a solution to 
 \begin{eqnarray*}
 \L_A(E) &=& \mu \qquad \mbox{in $ \Omega$} \\
 | \nabla E|_A &=&1 \qquad \mbox{a.e. in $ \Omega$} \\
 E &=&0 \qquad \mbox{on $ \pOm$}.
 \end{eqnarray*} 
 We then have the following theorem. 
 
 \begin{thm} Suppose $E$ is as above.  Then 
 \[ \int_\Omega | \nabla u|_A^2dx \ge \frac{1}{4} \int_\Omega \frac{u^2}{E^2}dx  + \frac{\lambda_0^2}{\|E\|_{L^\infty}^2} \int_\Omega u^2dx, \] for all $ u \in C_c^\infty(\Omega)$.
 
 \end{thm}

  \begin{proof} Let $ E$ be as above.   Now extend $E$ to all of $ \IR^n$ by setting $ E=0$ on $ \IR^n \backslash \Ov$, let $ E_\E$ denote the $ \E$ mollification of $E$ and $ F_\E:=\L_A(E_\E)$.  Returning to the proof of theorem \ref{exampleofimprovements} (ii) we have 
  
   \[
  \int_\Omega | \nabla  u |_A^2dx - \frac{1}{4} \int_\Omega \frac{ | \nabla E_\E|_A^2}{E_\E^2} u^2dx \ge \int_\Omega \frac{ | \nabla E_\E|_A^2}{f(E_\E)} \left( - f''(E_\E) - \frac{f'(E_\E)}{E_\E} \right) u^2dx +I_\E, \] where  
 \[ I_\E:= \int_\Omega \left( \frac{ f'(E_\E)}{f(E_\E)} + \frac{1}{2 E_\E} \right) u^2 F_\E dx,\]  
   for $ u \in C_c^\infty(\Omega)$ and $ 0 < f \in C^2( (0,\|E\|_{L^\infty}])$.   Now set $ \lambda:= \frac{ \lambda_0^2}{\|E\|_{L^\infty}^2}$ where $ \lambda_0$ is Lambs constant and define $f(t):=J_0(\sqrt{\lambda}t)$.   It is possible to show that 
   \[ f(t) >0, \qquad \frac{1}{f(t)} \left( - f''(t) - \frac{f'(t)}{t} \right) = \lambda, \qquad l(t):=\frac{f'(t)}{f(t)}+ \frac{1}{2t} \ge 0 \] in $ (0, \| E \|_{L^\infty})$.   Fixing $ u\in C_c^\infty(\Omega)$ and   subbing this $f$ into the above gives 
   \[ \int_\Omega | \nabla  u |_A^2dx - \frac{1}{4} \int_\Omega \frac{ | \nabla E_\E|_A^2}{E_\E^2} u^2dx \ge \frac{\lambda_0^2}{\|E\|_{L^\infty}^2} \int_\Omega | \nabla E_\E|_A^2  u^2 dx +I_\E,\] after noting that $ \| E_\E\|_{L^\infty} \le \|E \|_{L^\infty}$ and  where $I_\E:= \int_\Omega l(E_\E)u^2 F_\E dx$.    It is possible to show that $ l \in C^\infty( (0, \|E\|_{L^\infty} ])$.    A standard argument shows that $ l(E_\E) u \rightarrow l(E) u $ in $H_0^1(\Omega)$ and $ u F_\E dx \rightharpoonup u \mu $ in $H^{-1}(\Omega)$ and hence one can conclude that $ \liminf_{\E \searrow 0}I_\E \ge 0$.    Passing to the limit (as $ \E \searrow 0$) in the remaining integrals gives the desired result. 
  \end{proof}

   We now look at improvements of the weighted generalized Hardy inequalities.   The next theorem allows us to transfer our knowledge of improvements from the non-weighted case to the weighted case, at least in the case that $E$ is an interior weight.

  \begin{thm} \label{equiv} (\textbf{Weighted interior improvements}) Suppose $ E $ is an interior weight on $ \Omega$ and $ 0 \le V \in C^\infty( \Omega \backslash K)$.  Then the following are equivalent: \\
  (i) \quad  For all $ u \in H_0^1(\Omega)$ 
  \begin{equation} \label{A}
   \int_\Omega | \nabla u|_A^2dx \ge \frac{1}{4} \int_\Omega \frac{| \nabla E|_A^2}{E^2} u^2dx + \int_\Omega V u^2dx.
   \end{equation}
   
  (ii) \quad For all $ t \neq \frac{1}{2}$ and $ u \in X_t$ 
  \begin{equation} \label{B}
   \int_\Omega E^{2t} | \nabla u|_A^2dx \ge ( t - \frac{1}{2})^2 \int_\Omega | \nabla E|_A^2 E^{2t-2} u^2dx + \int_\Omega V E^{2t} u^2dx.
   \end{equation}
    
  (iii) \quad For all $ u \in X_\frac{1}{2} $ 
  \begin{equation} \label{C}
   \int_\Omega E | \nabla u|_A^2dx \ge \int_\Omega VE u^2dx.
   \end{equation}
  
  \end{thm}

   Using similar arguments one can obtain a version of theorem \ref{equiv} for the case when $ E $ is a boundary weight on $ \Omega$;  we omit the details since the results is not as clean.

  \begin{proof} Let $ E  $ be an interior weight on $ \Omega$ and $ 0 \le V \in C^\infty( \Omega \backslash K)$.  \\
  (i) $ \Rightarrow $ (ii) \\
  Suppose (i) holds, $ t < \frac{1}{2}$, $ u  \in C_c^{0,1}(\Omega \backslash K)$ and define $ v:=E^t u \in C_c^{0,1}(\Omega \backslash K)$.    Then putting $v$ into (\ref{A}) and performing some integration by parts gives (\ref{B}). \\
  (ii) $\Rightarrow$ (iii) \\
  Suppose (ii) holds.  Let  $ u \in C_c^{0,1}(\Omega \backslash K) $ which is an element of $ X_t$ for all $ t $.  Now using (\ref{B}) for this $u$ and sending $ t \nearrow \frac{1}{2}$ gives (\ref{C}).  \\
  (iii) $\Rightarrow$(i) \\
  Suppose (iii) holds, $ u \in C_c^{0,1}(\Omega \backslash K)$ and $ v:=E^\frac{-1}{2} u \in C_c^{0,1}(\Omega \backslash K)$.  Putting $v$ into (\ref{C}) and integrating by parts gives (\ref{A}) for all $ u \in C_c^{0,1}(\Omega \backslash K)$.
 
 \end{proof}

 \subsection{Hardy inequalities valid for $ u \in H^1(\Omega)$}
   
   Let $K$ be a compact subset of $ \Omega $ with $ dim_{box}(K) < n-2$.  Standard arguments show that $ C_c^{0,1}(\Ov \backslash K)$ is dense in $H^1(\Omega)$.    
   
  \begin{dfn} We say $E$ is a Neumann interior weight on $\Omega$ provided: there exists some compact $K \subset \Omega$, $dim_{box}(K)<n-2$, $ E \in C^\infty(\Ov \backslash K)$, $ \inf_\Omega E>0$, $\L_A(E)+E$ is a nonnegative nonzero measure $ \mu$ whose support is $K$, $ E=\infty$ on $K$ and $ A \nabla E \cdot \nu =0$ on $ \pOm$ where $ \nu(x)$ denotes the outward normal vector at $ x \in \pOm$.  
 
  \end{dfn}

  \begin{thm} \label{nonzero} Suppose $E$ is a Neumann interior weight on $\Omega$.  Then \\
  \flushleft
  (i) \quad For  $ u \in C_c^{0,1}(\Ov \backslash K) $  and $ v:=E^\frac{-1}{2} u $ we have 
  \begin{equation} \label{nz1}
  \int_\Omega | \nabla u|_A^2dx + \frac{1}{2} \int_\Omega u^2 dx \ge \frac{1}{4} \int_\Omega \frac{| \nabla E|_A^2}{E^2} u^2dx + \int_\Omega E | \nabla v|_A^2dx. 
  \end{equation} 
  
  (ii)  \begin{equation}  \label{OOO}
   \int_\Omega | \nabla u|_A^2dx + \frac{1}{2} \int_\Omega u^2dx \ge \frac{1}{4} \int_\Omega \frac{| \nabla E|_A^2}{E^2} u^2dx,
   \end{equation} holds for all $ u \in H^1(\Omega)$.  Moreover $ \frac{1}{4}$ and $ \frac{1}{2}$ are optimal in the sense that if one fixes $ \frac{1}{4}$ then you can do no better than $ \frac{1}{2}$ and vice-versa.    Also the inequality is not attained.

  \end{thm}

 One can again view the best constants in a different manner which is analogous to theorem \ref{alpha}; we omit the details.

\begin{proof} 
Let $E$ be a Neumann interior weight on $ \Omega$. \\ (i) \quad Let $ u \in C_c^{0,1}(\Ov \backslash K)$ and define $v:= E^\frac{-1}{2} u$. 
 Then \[
  | \nabla u|_A^2 = E | \nabla v|_A^2 + \frac{ | \nabla E|_A^2}{4 E^2} u^2 + v \nabla v \cdot A \nabla E, \]
  and integrating this over $ \Omega $ gives  
 \begin{equation} \label{opp}
\int_\Omega | \nabla u|_A^2 dx+ \frac{1}{2} \int_\Omega  u^2dx =\frac{1}{4} \int_\Omega \frac{ | \nabla E|_A^2}{E^2} u^2dx + \int_\Omega E | \nabla v|_A^2 dx.
\end{equation} 

(ii) \quad Using (i) and the fact that $ C_c^{0,1}(\Ov \backslash K)$ is dense in $H^1(\Omega)$ one obtains (\ref{OOO}) for all $ u \in H^1(\Omega)$. \\ We now show the constants are optimal. 
  We first show that $ E^t \in H^1(\Omega)$ for $ 0 < t < \frac{1}{2}$.   \\  As in the proof of lemma \ref{test} the following calculations are only formal but they can be justified as hinted at there; by first regularizing the measure, obtaining approximate solutions and passing to the limit.   Fix $ 0 < t < \frac{1}{2}$ and multiply $ \L_A(E)+E=\mu$ by $ E^{2t-1}$ and integrate over $ \Omega$ using integration by parts and the fact that $ E=\infty$ on $K$ along with the boundary conditions of $E$ to see that 
\begin{equation} \label{hhh}
\int_\Omega E^{2t}dx = (1-2t) \int_\Omega E^{2t-2} | \nabla E|_A^2dx,
\end{equation} which shows that $ E^t \in H^1(\Omega)$ for $ 0 < t < \frac{1}{2}$.   To show the constants are optimal we will use as a minimizing sequence $ E^t $ as $ t \nearrow \frac{1}{2}$. A computation shows
\[ \frac{\int_\Omega | \nabla E^t|_A^2dx + \frac{1}{2} \int_\Omega E^{2t} dx}{  \int_\Omega \frac{| \nabla E|_A^2}{E^2} E^{2t}  dx}= t^2+\frac{1}{2}-t, \] and we see that $ \frac{1}{4}$ is optimal.    One similarly shows $ \frac{1}{2}$ is optimal. \\ To show the inequality does not attain we, as usual, just hold on to the extra term that we dropped in the above calculations.  This term is positive for non-zero $ u \in H^1(\Omega)$ provided $ E^\frac{1}{2} \notin H^1(\Omega)$ which is the case after one considers (\ref{hhh}).

\end{proof}

  We now examine weighted versions of (\ref{OOO}).  Suppose $E$ is a Neumann interior weight on $ \Omega$  and as usual we let $K$ denote the support of $ \mu$.   For $ t \neq \frac{1}{2}$ and $ u \in C^{0,1}_c(\Ov \backslash K)$ we define 
  \begin{equation*}
   \| u \|_t^2:= \left\{
   \begin{array}{lr}
    \int_\Omega E^{2t} | \nabla u|^2dx + \int_\Omega E^{2t} u^2 dx & \qquad t < \frac{1}{2} \\
   \int_\Omega E^{2t} | \nabla u|^2dx  & \qquad t > \frac{1}{2},
   \end{array}
   \right.
   \end{equation*}
    and we let $Y_t$ denote the completion of $ C^{0,1}_c(\Ov \backslash K)$ with respect to this norm.   We then have the following theorem. 
  
  \begin{thm} Suppose $E$ is a Neumann interior weight on $ \Omega$  and $ t \neq \frac{1}{2}$.  Then 
  \[ \int_\Omega E^{2t}| \nabla u|_A^2dx + \left(\frac{1}{2}-t \right) \int_\Omega E^{2t} u^2 dx \ge \left( t - \frac{1}{2} \right)^2 \int_\Omega E^{2t-2} | \nabla E|_A^2 u^2 dx,\] for all $ u \in Y_t$.  Moreover the constants are optimal and not attained. 
  
    \end{thm} Note in particular that for $ t > \frac{1}{2}$ one only has a gradient term on the left hand side and so we can conclude that $ C^\infty(\Ov) $ is not contained in $ Y_t$ for $ t > \frac{1}{2}$.

  \begin{proof} Suppose $E$ is a Neumann interior weight on $ \Omega$, $ t \neq \frac{1}{2}$ and let $ u \in C_c^{0,1}(\Ov \backslash K)$.   Putting $ E^t u$ into (\ref{nz1}) gives  
  \[  \int_\Omega E^{2t} | \nabla u|_A^2dx + \left(\frac{1}{2}-t \right) \int_\Omega E^{2t} u^2 dx \ge \left( t - \frac{1}{2} \right)^2 \int_\Omega E^{2t-2} | \nabla E|_A^2 u^2 dx + \int_\Omega E | \nabla w|_A^2 dx \] where $ w:= E^{t - \frac{1}{2}} u$.    To show the constants are optimal one takes the same approach as in theorem \ref{weight1th}.  We now show the optimal constants are not obtained.   Suppose we have equality for some nonzero $u \in Y_t$.   Then it is easily seen that $ \sqrt{E} \in H^1(\Omega)$ which we know is not the case. 
  
  \end{proof}

   We now examine improvements of (\ref{OOO}). 
   \begin{thm} Suppose $E$ is a Neumann interior weight on $ \Omega$. Then  \\
   (i) \qquad Suppose $  V \in C^\infty( \Omega \backslash K)$ and there exists some $ 0 < \phi \in C^2(\Omega \backslash K) \cap C^1(\Ov \backslash K)$ such that 
   \begin{equation} \label{EQ1} -\L_A(\phi) + \frac{ A \nabla E \cdot \nabla \phi}{E} + V \phi \le 0 \qquad \mbox{ in $ \Omega \backslash K$},
   \end{equation} with $ A \nabla \phi \cdot \nu \ge 0$ on $ \pOm$.   Then 
   \[ \int_\Omega | \nabla u|_A^2 + \frac{1}{2} \int_\Omega u^2 dx - \frac{1}{4} \int_\Omega \frac{|\nabla E|_A^2}{E^2}u^2 dx \ge \int_\Omega V(x) u^2 dx, \] for all $ u \in H^1(\Omega)$.  \\
   (ii) \quad Suppose $ 0 \le V \in C^\infty(\Ov \backslash K)$ is such that 
   \[  \int_\Omega | \nabla u|_A^2 + \frac{1}{2} \int_\Omega u^2 dx - \frac{1}{4} \int_\Omega \frac{|\nabla E|_A^2}{E^2}u^2 dx \ge \int_\Omega V(x) u^2 dx, \] holds for all $ u \in H^1(\Omega)$.  In addition we assume that $ \{x \in \Omega: E(x)<t\}$ is connected for sufficiently large $t$.  Then there exists some $ 0 < \theta \in C^\infty(\Ov \backslash K)$ such that 
   \begin{equation} \label{EQ2} -\L_A(\theta)- \frac{\theta}{2} + \frac{| \nabla E|_A^2}{4E^2} \theta +V \theta \le 0 \qquad \mbox{in $ \Omega \backslash K$},
   \end{equation} with $ A \nabla \theta  \cdot \nu =0$ on $ \pOm$.

   \end{thm} 
   Note that one can go from (\ref{EQ1}) to (\ref{EQ2}) by using the change of variables $ \theta = \phi E^\frac{1}{2}$ in the case that $ A \nabla \phi \cdot \nu=0$ on $ \pOm$.

   \begin{proof} The proof is similar to the proof of theorem \ref{geninteriorimprov}. 
   
   \end{proof}

   \begin{remark} One can obtain an analogous  version of theorem \ref{equiv} for the case where $E$ is an interior weight on $ \Omega$ satisfying a Neumann boundary condition.

   \end{remark}

   \subsection{$H^1(\Omega)$ inequalities for exterior and annular domains}

   In this section we obtain optimal Hardy inequalities which are valid on exterior and annular domains.  Moreover these inequalities will be valid for functions $u$ which are nonzero on various portions of the boundary.   For simplicity we only consider the case where $A(x) $ is the identity matrix and hence $ \L_A = -\Delta$; the results immediately generalize to the case where $A(x)$ is not the identity matrix.     We first examine the exterior domain case.  \\ 
   
 \noindent  \textbf{Condition (Ext.):} \quad We suppose that $E>0$ in $ \IR^n$,  $ -\Delta E $ is a nonnegative nonzero finite measure (which we denote by $ \mu$) with compact support $K$ and we let $ \Omega $ denote a connected exterior domain in $ \IR^n$  with $ dist(K,\Omega)>0$.  In addition we assume that the compliment of $ \Omega$ denoted by $  \Omega^c$ is connected, $ \lim_{|x| \rightarrow \infty} E=0$ and $ \partial_\nu E \ge 0$ on $ \pOm$.   \\

   We will work in the following function space.  Let $ D^1(\Omega \cup \pOm) $ denote the completion of $ C_c^\infty(\Omega \cup \pOm)$ with respect to the norm $ \| \nabla u\|_{L^2(\Omega)}$.    Note we don't require $u$ to be zero on the boundary of $ \pOm$.   We then have the following theorem.  
   
 \begin{thm} \label{ext_theorem} Suppose $ E,\mu,K, \Omega$ are as in condition (Ext.).   Then \\
 \noindent (i) \quad For all $ u \in D^1(\Omega \cup \pOm)$ we have 
 \begin{equation} \label{exterior_domains}
 \int_\Omega | \nabla u|^2 dx \ge \frac{1}{4} \int_\Omega \frac{ | \nabla E|^2}{E^2} u^2 dx.
 \end{equation}  Moreover the constant is optimal and not attained.   \\
 (ii) \quad For all $ u \in D^1(\Omega \cup \pOm)$ we have 
 \begin{equation} \label{Ext0}
 \int_\Omega | \nabla u|^2dx \ge \frac{1}{4} \int_\Omega \frac{ | \nabla E|^2}{E^2} u^2 dx + \frac{1}{2} \int_{\pOm} \frac{u^2 \partial_\nu E}{E} dS(x).
 \end{equation}

 \end{thm}

 \begin{proof}  Let $ u \in C_c^\infty(\Omega \cup \pOm)$ and set $ v:= E^\frac{-1}{2} u$.   Then as before we have 
 \begin{equation} \label{Ext1} | \nabla u |^2 - \frac{| \nabla E|^2 u^2}{4E^2} = E | \nabla v|^2 + v \nabla v \cdot \nabla E, \qquad \mbox{in $ \Omega$.}
 \end{equation}
   Integrating the last term by parts  gives 
 \[ \int_\Omega v \nabla v \cdot \nabla E dx = \frac{1}{2} \int_{\pOm} \frac{ u^2 \partial_\nu E }{E} dS(x).\]
 
 We obtain (\ref{Ext0}) by integrating (\ref{Ext1}) over $ \Omega$ and since $ \partial_\nu E \ge 0 $ on $ \pOm$ we obtain (\ref{exterior_domains}).   We now show the constant is optimal.   For big $ R$ we set $ \Omega_R:=\Omega \cap B_R$ where $ B_R$ is the ball centered at $0$ with radius $R$.  Let $ \frac{1}{2} <t<1$ and multiply $ -\Delta E = \mu$ by $ E^{2t-1}$ and integrate over $ \Omega_R$ to obtain 
 \[ (2t-1) \int_{\Omega_R} E^{2t-2} | \nabla  E|^2 dx = \int_{\pOm} \partial_\nu E E^{2t-1} dS(x) + \int_{\partial B_R} \partial_\nu E E^{2t-1} dS(x).\] Using a Newtonian potential argument one can show that as $ R \rightarrow \infty$ the surface integral over the ball $ B_R$ goes to zero.    So using this one sees that 
 \begin{equation} \label{Ext5} (2t-1) \int_\Omega E^{2t-2} | \nabla E|^2 dx = \int_{\pOm} \partial_\nu E E^{2t-1} dS(x),
 \end{equation}
  and so $ \int_\Omega | \nabla E^t|^2 dx <\infty$.  With this along with a standard cut-off function argument one sees that $ E^t \in D^1(\Omega \cup \pOm)$.  Now one uses $ E^t$ as $ t \searrow \frac{1}{2}$ as a minimizing sequence to show that $ \frac{1}{4}$ is optimal.  We now show the constant is not attained.  Now assume that $ x_0 \in \pOm$ is such that $ E(x_0) = \min_{\pOm} E$.   Then by Hopf's lemma $ \partial_\nu E(x_0) >0$ and so using this along with continuity and (\ref{Ext5}) one sees that $ E^\frac{1}{2} \notin D^1(\Omega \cup \pOm)$.    Now to finish the proof it will be sufficient to show that 
  \[ \int_\Omega E | \nabla v|^2 dx >0 \] for all nonzero $ u \in D^1(\Omega \cup \pOm)$.     The only nonzero $u$'s for which this integral is zero are multiples of $ E^\frac{1}{2}$ which are not in $D^1(\Omega \cup \pOm)$.

 \end{proof} 
 
 \begin{exam} Take $ \Omega $ a exterior domain in $ \IR^n$ where $ n \ge 3$, $ 0 \notin \overline{\Omega}$, and such that $ \nu(x) \cdot x \le 0$ on $ \pOm$ where $ \nu(x)$ is the outward pointing normal.  Define  $ E(x):=|x|^{2-n}$ and use theorem \ref{ext_theorem} to see that 
\begin{equation} \label{ext_exam}
 \int_\Omega | \nabla u|^2 dx \ge \left( \frac{n-2}{2} \right)^2 \int_\Omega \frac{u^2}{|x|^2} dx,
 \end{equation}
  for all $ u \in D^1(\Omega \cup \pOm)$.  Moreover the constant is optimal and not attained.   In fact using (ii) from the same theorem shows we can add the following nonnegative term to the right hand side of (\ref{ext_exam}):
 \[ \frac{(n-2)}{2} \int_{\pOm} \frac{u^2 ( - x \cdot \nu)}{|x|^2} dS(x).\]

 \end{exam}
 
 We now examine the annular domain case.   \\
 
 \noindent \textbf{Condition (Annul.):}  We assume that $ \Omega_1 \subset \subset \Omega_2$ are two bounded connected domains in $ \IR^n$ with smooth boundaries and $ \Omega:=\Omega_2 \backslash \overline{ \Omega_1}$ is connected.  In addition we assume that $ E>0 $ in $ \Omega_2$ with $ -\Delta E = \mu$ in $ \Omega_2$ where $ \mu$ is a nonnegative nonzero finite measure supported on $K \subset \Omega_1$. We also assume that $ \partial_\nu E \le 0 $ on $ \pOm_1$.
 
 We then have the following theorem. 
 
\begin{thm} Suppose $ \Omega, K,E$ are as in condition (Annul.).   Then \\
\noindent 
(i) \quad For all $ u \in H^1(\Omega)$ with $ u=0$ on $ \pOm_2$ we have
\begin{equation} \label{Ext10}
\int_\Omega | \nabla u|^2dx \ge \frac{1}{4} \int_\Omega \frac{| \nabla E|^2}{E^2}u^2 dx.
\end{equation} Moreover the constant is optimal and not attained if we assume that $ E=0$ on $ \pOm_2$. \\
(ii) \quad For all $ u \in H^1(\Omega)$ with $ u =0$ on $ \pOm_2$ we have 
\begin{equation} \label{Ext11}  \int_\Omega | \nabla u|^2dx \ge \frac{1}{4} \int_\Omega \frac{ | \nabla E|^2}{E^2} u^2 dx + \frac{1}{2} \int_{\pOm} \frac{u^2 \partial_\nu E}{E} dS(x).
\end{equation}

\end{thm}

   \begin{proof}  The proof of (\ref{Ext10}) and (\ref{Ext11}) is similar to the previous theorem so we omit the details.    We now show the constant is optimal.   Let $ H_0^1(\Omega \cup \pOm_1)$ denote $ \{ u \in H^1(\Omega): u = 0 \; \mbox{on} \; \pOm_2 \}$.  Again we multiply $ -\Delta E = \mu$ by $ E^{2t-1}$ for $ \frac{1}{2}< t <1$ and integrate over $ \Omega$ to obtain 
   \[ (2t-1) \int_\Omega E^{2t-2} | \nabla E|^2 dx = - \int_{\pOm_1} \partial_\nu E E^{2t-1} dS(x), \] which shows that $ E^t \in H_0^1(\Omega \cup \pOm_1)$.   From this one obtains 
   \[ \lim_{t \searrow \frac{1}{2} }(2t-1) \int_\Omega E^{2t-2} | \nabla E|^2 dx = \mu(\Omega_1)>0,\] which shows that $ E^\frac{1}{2} \notin H_0^1(\Omega \cup \pOm_1)$.    To see the constant is optimal one uses the same minimizing sequence as in the previous theorem.  To see the constant is not attained one uses the fact that $ E^\frac{1}{2} \notin H_0^1(\Omega \cup \pOm_1)$.

   \end{proof}

  \begin{remark} These inequalities have analogous weighted versions and using the methods developed earlier one easily obtains results concerning improvements.  We leave this for the reader to develop.

  \end{remark}

 \subsection{The non-quadratic case}

   For $ 1 < p \le n$ we define $ \L_{A,p}(E):= - div( | \nabla E|_A^{p-2} A \nabla E)$. As mentioned earlier 
    Adimurthi and Sekar [AS] obtained generalized Hardy inequalities of the form  
   \begin{equation} \label{lll} \int_\Omega | \nabla u|_A^pdx - \left( \frac{p-1}{p} \right)^p \int_\Omega \frac{| \nabla E|_A^p}{E} |u|^p dx\ge 0,
    \end{equation} where $ u \in W_0^{1,p}(\Omega)$.     There approach (as their title suggests) was to look at functions $ E $ which solve 
    \begin{eqnarray*}
    \L_{A,p}(E) &=& \delta_0 \qquad \mbox{ in $ \Omega$} \\
    E&=& 0 \qquad \mbox{ on $ \pOm$},
    \end{eqnarray*}
    where $ 0 \in \Omega$ and where  $\delta_0$ is again the Dirac mass at $0$.  \\ They posed the question (see [AS]) as to whether  $  (\frac{p-1}{p})^p$ is optimal in (\ref{lll})?    
   The next theorem shows this is the case (at least for $1 <p<n$); infact we show the result for a more general case.

   \subsubsection*{Interior case}

  Suppose $ \mu$ is a  nonnegative nonzero finite measure supported on $K \subset \Omega $, $ dim_{box}(K) < n-p$ (and hence $C_c^{0,1}(\Omega \backslash K)$ is dense in $W_0^{1,p}(\Omega)$)   and $ 0 < E $ is a solution of 
  \begin{equation} \label{peq}
  \L_{A,p}(E) = \mu \qquad \mbox{ in $ \Omega$}.
  \end{equation}  
  By regularity theory (see [D], [T]) there is some  $ 0 < \sigma <1 $ such that $ E \in C^{1,\sigma}(\Omega \backslash K)$ and by the maximum principle (see [V]) $ E >0 $ in $ \Omega \backslash K$.   Now if we assume that $ \mu = \delta_0$, as was the case in the question posed in [AS], then one can show $E(0)=\infty$.

  \begin{thm}  \label{non-linear-hardy} Suppose $E$ is as above but we don't assume that $ E=\infty$ on $K$.  \\
  (i) \quad Then 
  \begin{equation} \label{bb}
   \int_\Omega | \nabla u|_A^2dx \ge \left( \frac{ p-1}{p} \right)^p \int_\Omega \frac{ | \nabla E|_A^p}{E^p} | u |^p dx, 
  \end{equation} for all $ u \in W_0^{1,p}(\Omega)$.   \\ 
  (ii) \quad Suppose $ E = \infty$ on $K$ and $ E= \gamma $ on $ \pOm$ where $ \gamma $ is a non-negative constant.   Then the constant in (\ref{bb}) is optimal.

  \end{thm}

\begin{proof}  (i) \quad Let $ u \in C_c^{0,1}(\Omega \backslash K)$.  Then $ \nabla E^{1-p} = (1-p) E^{-p} \nabla E $ and dotting both sides with $ | \nabla E|_A^{p-2} A \nabla E | u|^p $ and integrating over $ \Omega $ gives
  \begin{eqnarray*}
  (1-p) \int_\Omega \frac{| \nabla E|_A^p}{E^p} | u|^pdx & = & \int_\Omega \nabla E^{1-p} \cdot \left( | \nabla E|_A^{p-2} A \nabla E |u|^p \right) dx\\ 
  & =& \int_\Omega E^{1-p} |u|^p d \mu \\
  && - \int_\Omega E^{1-p} | \nabla E|_A^{p-2} A \nabla E \cdot p |u|^{p-2} u \nabla udx \\
  &=&- \int_\Omega E^{1-p} | \nabla E|_A^{p-2} A \nabla E \cdot p |u|^{p-2} u \nabla u dx,
  \end{eqnarray*} where we used the divergence theorem and also the fact that $u=0$ on $K$.   Now using the Cauchy-Schwarz inequality on the inner product induced by $ A(x)$ we see that 
  \[ \frac{p-1}{p} \int_\Omega \frac{ | \nabla E|_A^p}{E^p} | u|^pdx \le \int_\Omega \frac{ | \nabla E|_A^{p-1} |u|^{p-1}}{E^{p-1}} | \nabla u|_A dx, \] and we now apply Holder's inequality on the right after recalling $ (p-1)p'=p$ where $ p'$ is the conjugate of $p$.   Now use density to extend to all of $ W_0^{1,p}(\Omega)$.  \\
  (ii) \quad  We first consider the case $ \gamma >0$.  We begin by showing that  $ u_t:=E^t - \gamma^t \in W^{1.p}_0(\Omega)$ for $ 0 < t < \frac{p-1}{p}$.   Fix $  0 < t < \frac{p}{p-1}$ and multiply (\ref{peq}) by $E^{tp-p+1}$ and integrate over $ \Omega $ to get 
  \begin{eqnarray*}
  0 = \int_\Omega E^{tp-p+1} d \mu &=& (tp-p+1) \int_\Omega | \nabla E|_A^2 E^{tp-p}dx - \gamma^{tp-p+1} \int_{\pOm} | \nabla E|_A^{p-2} A \nabla E \cdot \nu d \mathcal{H}^{n-1} \\
 &=& (tp-p+1) \int_\Omega | \nabla E|_A^2 E^{tp-p}dx - \gamma^{tp-p+1} \int_\Omega div( | \nabla E|_A^{p-2} A \nabla E)dx \\
 &=& (tp-p+1) \int_\Omega | \nabla E|_A^2 E^{tp-p}dx + \gamma^{tp-p+1} \mu(\Omega),
 \end{eqnarray*} where the first integral is zero since $ E = \infty$ on $K $ and $ tp-p+1 <0$.    Re-arranging this we arrive at 
 \[ \int_\Omega | \nabla E^t|_A^pdx= \frac{ \mu(\Omega) \gamma^{tp-p+1} t^p}{p-tp-1}, \] from which we see that $ E^t \in W^{1,p}(\Omega)$ for $ 0 <t < \frac{p-1}{p}$ and we also see that  
 \[ \lim_{t \nearrow \frac{p-1}{p}} \int_\Omega | \nabla E|_A^p E^{tp-p} dx = \infty.\]  Put $ t $ as above and set $ u_t:=E^t-\gamma^t \in W_0^{1,p}(\Omega)$.  By the binomial theorem we have 
 \[ (1+x)^p = \sum_{m=0}^\infty ( p,m) x^m, \] for all $ |x| \le 1$ where  $(p,m)$ are the binomial coefficients.    One should note that $ (p,m)$ is eventually alternating and since we have convergence at $ x=-1$ we see  that $ \sum_m (p,m) (-1)^m$  converges; which shows that $ \sum_m | (p,m) | < \infty$.   Now we have
 \begin{eqnarray*}
 | u_t|^p &=& E^{tp} \big| 1 - \frac{\gamma^t}{E^t} \big| \\
 &=& E^{tp} \sum_{m=0}^\infty (p,m) \frac{ (-1)^m \gamma^{tm}}{E^{tm}},
 \end{eqnarray*}    and we define 
 \[ Q_t:= \frac{ \int_\Omega \frac{| \nabla E|_A^p}{E^p} |u_t|^pdx }{   \int_\Omega | \nabla u_t|_A^pdx}. \] So
 
 \[ Q_t - \frac{1}{t^p} = \frac{ \int_\Omega | \nabla E|_A^p E^{tp-p} \left( \sum_{m=1}^\infty (p,m) (-1)^m \frac{\gamma^{tm}}{{E^{tm}}} \right) dx  }{ t^p \int_\Omega | \nabla E|_A^p E^{tp-p} dx }, \] and so 
 \begin{eqnarray*}
 \big| Q_t - \frac{1}{t^p} \big| & \le & \frac{1}{t^p} \sum_{m=1}^\infty \big| (p,m) \big| \frac{ \int_\Omega | \nabla E|_A^p E^{tp-p} \gamma^{tm} E^{-tm} dx }{   \int_\Omega | \nabla E|_A^p E^{tp-p} dx } \\
 & = &  \frac{1}{t^p} \sum_{m=1}^\infty \big| (p,m) \big| \frac{ p-tp-1}{p-tp-1+tm} \\
 & \le & \frac{ p -tp-1}{t^{p+1}} \sum_{m=1}^\infty \frac{ \big| (p,m) \big|}{m} \\
 &=:& \frac{ p -tp-1}{t^{p+1}} C_p,
 \end{eqnarray*}  and so we see that 
 \[ \lim_{t \nearrow \frac{p-1}{p}} \big| Q_t - \frac{1}{t^p} \big| =0, \] which shows the constant in (\ref{bb}) is optimal.  \\ Now we handle the case $ \gamma=0$.  Let $ \L_{A,p}(E) = \mu$ in $ \Omega $ and $ E=0$ on $ \pOm$ and define $ E_\E:=\E + E $ where $ \E>0$.   Then $ \L_{A,p}(E_\E)= \mu$ in $ \Omega $ and $ E_\E = \E $ on $ \pOm$.  For $ u \in C_c^\infty(\Omega)$ non-zero we have, after some simple algebra,
 \[ \frac{ \int_\Omega | \nabla u|_A^pdx}{ \int_\Omega \frac{| \nabla E|_A^p}{E^p} |u|^pdx} \le  \frac{ \int_\Omega | \nabla u|_A^pdx}{ \int_\Omega \frac{| \nabla E_\E|_A^p}{E_\E^p} |u|^pdx}, \] which shows the constant is optimal in the case of $ \gamma =0$.   

\end{proof}

  \subsubsection*{Boundary case}
  
  Analogously to the quadratic case we will be interested in the validity of (\ref{bb}) when $E$ is a solution to 
  \begin{eqnarray*}
  \L_{A,p}(E) &=& \mu \qquad \mbox{ in $ \Omega$,} \\
  E &=& 0 \qquad \mbox{ on $ \pOm$}
  \end{eqnarray*} where $ \mu$ is a nonnegative nonzero finite measure and where we impose some added regularity restrictions to $E$ or $ \mu$.   Recall in the quadratic case we  added the condition that $ E \in H_0^1(\Omega)$.    For simplicity we will assume that $ \mu$ is smooth; say $ d \mu = f dx $ where $ 0 \le f \in C^\infty(\Ov)$ is non-zero.   One can show that $ E \in C^{1,\sigma}(\Ov)$ for some $ 0 < \sigma <1$. 
 
 \begin{thm} \label{non-lin-bound} Suppose $ E$ is a positive solution to $ \L_{A,p}(E)=\mu $ in $ \Omega $ where $ \mu$ is as above. \\
 (i) \qquad  Then 
 \begin{equation} \label{p-measure}
 \left( \frac{p-1}{p} \right)^p \int_\Omega \frac{ | \nabla E|_A^p}{E^p} |u|^pdx + \left( \frac{p-1}{p} \right)^{p-1} \int_\Omega \frac{|u|^p}{E^{p-1}} d \mu \le \int_\Omega | \nabla u|_A^pdx, 
 \end{equation}
 for all $ u \in W_0^{1,p}(\Omega)$.  Since $ \mu$ is a measure we have 
 \begin{equation} \label{p-nomeas}
 \left( \frac{p-1}{p} \right)^p \int_\Omega \frac{ | \nabla E|_A^p}{E^p} |u|^pdx \le \int_\Omega | \nabla u|_A^pdx,
 \end{equation}
 for all $ u \in W_0^{1,p}(\Omega)$. \\
 (ii) \quad Suppose $ E =0 $ on $ \pOm$.  Then (\ref{p-nomeas}) is optimal.  \\
 (iii) \quad Suppose $ E=0$ on $ \pOm$.   If one fixes the optimal constant from part (ii) then the other constant is also optimal in (\ref{p-measure}) ie. 
 \[ \inf \left\{ \frac{ \int_\Omega | \nabla u|_A^pdx  - \left( \frac{p-1}{p} \right)^p \int_\Omega \frac{ | \nabla E|_A^p}{E^p} |u|^pdx}{ \int_\Omega \frac{|u|^p}{E^{p-1}} d \mu } :  u \in W_0^{1,p}(\Omega), u \neq 0 \right\} = \left( \frac{p-1}{p} \right)^{p-1}. \]
 
  \end{thm}

    \begin{proof}
 (i) \quad Suppose $ E $ is a positive solution to $ \L_{A,p}(E) = \mu $ in $ \Omega$ and let $ u \in C_c^\infty(\Omega)$.  From the proof of theorem \ref{non-linear-hardy} we have 
  \begin{eqnarray*}
  (p-1) \int_\Omega \frac{ | \nabla E|_A^p}{E^p}|u|^pdx + \int_\Omega \frac{|u|^p}{E^{p-1}} d \mu & = & p \int_\Omega \frac{ | \nabla E|_A^{p-2}}{E^{p-1}} A \nabla E \cdot \nabla u |u|^{p-2} u dx \\
  & \le & p \left( \int_\Omega \frac{| \nabla E|_A^p}{E^p} |u|^p dx \right)^\frac{1}{p'} \left( \int_\Omega | \nabla u|_A^p dx \right)^\frac{1}{p}.  
  \end{eqnarray*}   Now let $ q $ denote $ p'$ and 
  \[ B:= \int_\Omega \frac{ | \nabla E|_A^p}{E^p}|u|^pdx, \qquad C:=\int_\Omega \frac{|u|^p}{E^{p-1}} d \mu, \qquad D:= \int_\Omega | \nabla u|_A^p dx.\]   Using Young's inequality with $ t >0 $ we arrive at 
 \[
  \frac{(p-1)}{p} B + \frac{C}{p} \le  B^\frac{1}{q} D^\frac{1}{p} \le   t B + C(t) D,  \] 
  where 
  \[ C(t):= p^{-1} q^\frac{-p}{q} t^\frac{-p}{q}, \] and so 
   \[ \frac{1}{C(t)} \left( \frac{p-1}{p} -t \right) B + \frac{1}{pC(t)} C \le D, \] for all $ t >0$.   Picking $ t=q^2$ gives the desired result.  \\
   (ii) \quad Let $  t> \frac{p-1}{p}$, multiply $ \L_{A,p}(E)= \mu$ by $ E^{tp-p+1}$ and integrate over $ \Omega $ to obtain
  \begin{equation} \label{must}
   \int_\Omega E^{tp-p+1} d \mu = ( tp-p+1) \int_\Omega  | \nabla E|_A^p E^{tp-p}dx, 
   \end{equation} which shows that $ E^t \in W_0^{1,p}(\Omega)$ for $ p > \frac{p-1}{p}$.  If one uses as a minimizing sequence $ u_t:=E^t$ and sends $ t \searrow \frac{p-1}{p}$ they immediately see  that (\ref{p-nomeas}) is optimal. \\
  (iii)  Again one uses $ u_t:=E^t$ and sends $ t \searrow \frac{p}{p-1}$.  The result is immediate after using (\ref{must}). 
   
\end{proof}

 An important example is when $A(x)$ is the identity matrix and $ E(x)=\delta(x):=dist(x,\pOm)$ so $ | \nabla \delta|=1$ a.e.. Then $ \L_{A,p}(\delta)= -div( | \nabla \delta |^{p-2} \delta ) = - \Delta \delta =:\mu$ which is non-negative if we further assume that $ \Omega $ is convex.   In this case we have the $L^p$ analog of (\ref{bound-HS}): 
 \begin{cor} Suppose $ \Omega $ is convex and $\delta(x):=dist(x,\pOm)$.   Then for $ 1<p<\infty$ and $ u \in W_0^{1,p}(\Omega)$ we have  
 \[ \int_\Omega | \nabla u|^p dx \ge  \left( \frac{p-1}{p} \right)^p \int_\Omega \frac{|u|^p}{\delta^p}dx,\] 
 \[ \int_\Omega | \nabla u|^pdx- \left( \frac{p-1}{p} \right)^p \int_\Omega \frac{|u|^p}{\delta^p}dx \ge \left( \frac{p-1}{p} \right)^{p-1} \int_\Omega \frac{ |u|^p}{\delta^{p-1}} d \mu, \] where $ d \mu:= -\Delta \delta dx$.  Moreover all constants are optimal.

 \end{cor}
 One can view the second inequality as an improvement of the first.

\end{document}